\documentclass[review,hidelinks,onefignum,onetabnum]{siamart220329}

\usepackage{lipsum}
\usepackage{amsfonts}
\usepackage{epstopdf}
\usepackage{algorithmic}
\usepackage{tikz-cd}
\usepackage{subfigure}
\usepackage{amssymb,amsmath,graphicx,amscd,mathrsfs}
\usepackage{color,xcolor,amsmath}
\usepackage{bm}
\usepackage{stmaryrd}
\usepackage{cite}

\ifpdf
  \DeclareGraphicsExtensions{.eps,.pdf,.png,.jpg}
\else
  \DeclareGraphicsExtensions{.eps}
\fi

\def\bw{{\mathbf{w}}}
\def\bu{{\mathbf{u}}}
\def\bv{{\mathbf{v}}}
\def\bn{{\mathbf{n}}}

\def\bf{{\mathbf{f}}}

\def\curl{{\operatorname{curl}}}

\newtheorem{remark}{Remark}[section]

\newtheorem{assumption}{Assumption}[section]

\def\3bar{{|\hspace{-.02in}|\hspace{-.02in}|}}

\numberwithin{equation}{section}

\def\bn{\boldsymbol{n}}
\def\bu{\boldsymbol{u}}

\def\bH{\boldsymbol{H}}
\def\bw{\boldsymbol{w}}

\def\bv{\boldsymbol{v}}
\def\bxi{\boldsymbol{\xi}}

\def\d{{\mathrm d}}
\def\Omega{{\varOmega}}

\newsiamremark{hypothesis}{Hypothesis}
\crefname{hypothesis}{Hypothesis}{Hypotheses}
\newsiamthm{claim}{Claim}

\title{Unfitted finite element method for the quad-curl interface problem}

\author{ Hailong Guo\thanks{School of Mathematics and Statistics, The University of Melbourne, Parkville, VIC 3010, Australia(\email{hailong.guo@unimelb.edu.au}).}
\and Mingyan Zhang\thanks{Corresponding author. Beijing Computational Science Research Center, Beijing 100193, China(\email{myzhang@csrc.ac.cn}).}
\and Qian Zhang\thanks{Department of Mathematical Sciences, Michigan Technological University, Houghton, MI 49931, USA 
 (\email{qzhang15@mtu.edu}).}
\and Zhimin Zhang\thanks{Department of Mathematics, Wayne State University, Detroit, MI 48202, USA(\email{ag7761@wayne.edu}).}
}

\usepackage{amsopn}

\ifpdf
\hypersetup{
  pdftitle={Unfitted finite element method for the quad-curl interface problem},
}
\fi

\graphicspath{{fig/}{fig/mesh/}}

\begin{document}
\nolinenumbers
\maketitle
\begin{abstract}
In this paper, we introduce a novel unfitted finite element method to solve the quad-curl interface problem. We adapt Nitsche's method for curlcurl-conforming elements and double the degrees of freedom on interface elements. To ensure stability, we incorporate ghost penalty terms and a discrete divergence-free term. We establish the well-posedness of our method and demonstrate an optimal error bound in the discrete energy norm. We also analyze the stiffness matrix's condition number.  Our numerical tests back up our theory on convergence rates and condition numbers.
\end{abstract}



\section{Introduction}

The quad-curl problem presents a significant challenge in multiphysics simulations, especially in magnetohydrodynamics (MHD) when modeling magnetized plasmas \cite{Biskamp1996magnetic,Zheng2011nonconforming}. It also plays a pivotal role in approximating the Maxwell transmission eigenvalue problem within inverse electromagnetic scattering theory \cite{Monk2012finite,Cakoni2007variational,Cakoni2010inverse}. Over the years, numerous numerical methods have been devised to address the quad-curl problem. For instance, mixed schemes have been widely explored in the literature \cite{Monk2012finite, Sun2016mixed, Zhang2018mixed, Wang2019new}. Hong et al. introduced a discontinuous Galerkin (DG) method based on the N\'ed\'elec finite element space \cite{Hong2012discontinuous}. Furthermore, nonconforming and low-order finite element spaces have been developed \cite{Zheng2011nonconforming,Huang2023nonconforming,Zhang2022nonconforming}. Recently, some researchers, including one of the authors, explored the construction of conforming finite element spaces, as referenced in \cite{Hu2022family,HuZhangZhang2020,Zhang2019h,Zhang2020family}. It is worth noting that authors in \cite{HuZhangZhang2020} construct a family conforming finite element including complete polynomials and the lowest-order element has only $13$ degrees of freedom.

In scientific and engineering research, addressing interface problems is of fundamental importance. Such problems often feature discontinuous coefficients within the partial differential equations, evident at the interface. Numerous numerical methods for interface problems with arbitrary but smooth interfaces have been extensively studied during the last three decades.  Broadly, these methods fall into two categories: fitted mesh methods and unfitted methods. While fitted mesh methods have been investigated for elliptic interface problems as seen in \cite{Babuvska1970finite, Chen1998finite, Barrett1987fitted, Xu1982}, generating body-fitted meshes, especially when dealing with complex and moving interfaces, remains a notably challenging and time-consuming task. 

To mitigate the challenge inherent in mesh generation, considerable attention has been directed towards unfitted methods, wherein the interface is allowed to intersect the mesh. Such methods have gained significant interest since the pioneering work of Peskin \cite{Peskin1977}, where he proposed a first-order approach to approximate delta functions using smoother functions. To improve convergence rates, Leveque and Li \cite{Leveque1994immersed} proposed the immersed interface methods which utilize interface conditions to model discontinuities and adapt them to design special finite difference schemes at grid points near the interfaces. This concept was subsequently extended to the finite element framework in \cite{Li1998immersed,LiLinWu2003}, where basis functions were tailored to satisfy homogeneous interface conditions. 

Another widely adopted unfitted method is the Cut Finite Element Method (CutFEM), originally proposed by Hansbo et al. in \cite{Hansbo2001}, building upon Nitsche's penalty method. CutFEM effectively addresses elliptic interface problems by doubling the degrees of freedom on interface elements and introducing penalty terms to weakly enforce transmission conditions across the interface. Researchers have successfully applied this method across various domains, including elasticity interface problems \cite{Hansbo2004finite}, Stokes interface problems \cite{Hansbo2014}, $H(\text{curl})$ and $H(\text{div})$ interface problems \cite{Liu2020interface,Li2023reconstructed}, biharmonic interface problems \cite{Cai2021nitsche}, and for interface eigenvalue problems \cite{GuoYangZ2021, GuoYangZhu2021}.  For a thorough discussion on CutFEM, readers are referred to the survey paper    \cite{Burman2015}.

From our thorough literature review, quad-curl interface problems have yet to be addressed.
Such problems arise in Magnetohydrodynamics and Maxwell transmission problems when dealing with heterogeneous materials. In light of this gap, the objective of this paper is to develop an unfitted finite element method, i.e. CutFEM, for the quad-curl interface problem, and in the meantime establish a complete theory of existence, uniqueness, and convergence of numerical solutions. Three primary challenges are addressed in the formulation and analysis of our numerical method. First, we introduce two ghost penalty terms (c.f. \cite{Massing}) and $\sum_{K\in\mathcal{T}_h}\int_{K}h^{-2}\nabla\cdot\bu_h \nabla\cdot\bv_h \ dx$ for improving the stability of numerical method and ensuring the divergence-free condition. 
Second, the interpolation estimate in the discrete energy norm (see Lemma \ref{InterpolationEstimate}) and the weak Galerkin orthogonality relation (see Lemma \ref{Weakorthogonality}) are established. By combining these results with the estimate for the ghost penalty (see Lemma \ref{WeakConsistency}), we derive error estimates between the numerical solution $\mathbf{u}_h$ and the exact solution $\mathbf{u}$ in the discrete energy norm.
Finally, we present estimates between the degree of freedom vector $W$ and function $\bw_h$ for curlcurl-conforming finite element (see Lemma \ref{Vectortofunc}), which differ from the estimates obtained using Lagrange elements. By employing Poincar\'e-type inequalities, as demonstrated in Lemma \ref{poincare}, we can obtain the condition number of the stiffness matrix.

The rest of this paper is organized as follows. In Section \ref{Sec:pre}, we introduce some fundamental notation related to Sobolev space and propose the quad-curl interface problem along with its variational formulation. In Section \ref{Sec:discrete}, we present the discrete weak formulation and the lowest order curl-curl conforming finite element space containing piecewise $P_2$ polynomials. The existence and uniqueness of the discrete scheme are demonstrated in Section \ref{Sec:wellposedness}. Subsequently, in Section \ref{Sec:error}, we establish error estimates for the numerical schemes in terms of the energy norm, while Section \ref{Sec:conditionnumber} provides the condition number estimate for the stiff matrix. To validate the accuracy of the proposed method, a series of numerical tests is presented in Section \ref{Sec:Numerical}. Finally, the main result is summarized and concluded in Section \ref{Sec:conclusion}.

\section{Preliminaries}\label{Sec:pre} In this section, we provide necessary background materials. In subsection \ref{ssec:not}, we introduce some notation used throughout the paper. In subsection \ref{ssec:interproblem}, we define the quad-curl interface problem. Finally, in subsection \ref{ssec:weakform}, we derive the weak form for the quad-curl interface problem and demonstrate its well-posedness.

\subsection{Notation}\label{ssec:not} Let $\Omega\subset \mathbb{R}^2$ be a bounded convex polygonal domain with Lipschitz boundary $\partial \Omega$.  Throughout this paper, we adopt the standard notation for Sobolev spaces and their associate norms given in \cite{BrennerScott2008, Ci2002, Evans2008}. 
Let $D$ be a subdomain of $\Omega$. We denote by $H^m(D)$ the Sobolev spaces equipped with the norm $|\cdot|_{m,D}$ and $H^m_0(D)$ the subspace characterized by functions with vanishing trace on the boundary $\partial D$. Specifically, when $m=0$, the space $H^0(D)$ coincides with $L^2(D)$, which is endowed with the conventional $L^2$-inner product denoted by $(\cdot,\cdot)_D$, along with the associated norm $\|\cdot\|_D$.  When dealing with the entire domain $\Omega$, the subscript $\Omega$ is omitted for simplicity.  Furthermore, we utilize $\boldsymbol{H}^m(D)$ and $\boldsymbol{L}^2(D)$ to represent the vector-valued Sobolev spaces $(H^m(D))^2$ and  $(L^2(D))^2$ respectively.

Consider vector-valued functions $\boldsymbol{u}=(u_1,u_2)^{\mathrm T}$ and $\boldsymbol{w}=(w_1,w_2)^{\mathrm T}$ whose cross product is defined as $\boldsymbol{u}\times\boldsymbol{w}=u_1w_2-u_2w_1$. We introduce the curl operator by $\nabla\times \bm u=\partial u_2/\partial x_1-\partial u_1/\partial x_2$ and the divergence operator by $\nabla\cdot\bm u=\partial u_1/\partial x_1+\partial u_2/\partial x_2$. Additionally, for a scalar function $v$, we define the curl of the scalar function by $\bm{\nabla}\times v=(\partial v/\partial x_2,-\partial v/\partial x_1)^{\mathrm T}$. 

For the quad curl problem, we require the following function spaces 
\begin{equation*}
\begin{aligned}
	H(\text{curl}^2;D) &:= \left\{\boldsymbol{u} \in \boldsymbol{L}^2(D): \nabla\times \boldsymbol{u} \in L^2(D), \ \bm{\nabla}\times\nabla\times\boldsymbol{u} \in \boldsymbol{L}^2(D) \right\}, \\
	H(\text{div};D) &:= \left\{\boldsymbol{u} \in \boldsymbol{L}^2(D): \nabla\cdot \boldsymbol{u} \in L^2(D) \right\},
\end{aligned}
\end{equation*}
where their inner products are given by
\begin{equation*}
\begin{aligned}
	(\boldsymbol{u},\boldsymbol{v})_{H(\text{curl}^2;D)} &:= (\boldsymbol{u},\boldsymbol{v})_D + (\nabla\times\boldsymbol{u},\nabla\times\boldsymbol{v})_D + (\bm{\nabla}\times\nabla\times\boldsymbol{u},\bm{\nabla}\times\nabla\times\boldsymbol{v})_D, \\
	(\boldsymbol{u},\boldsymbol{v})_{H(\text{div};D)} &:= (\boldsymbol{u},\boldsymbol{v})_D + (\nabla\cdot\boldsymbol{u},\nabla\cdot\boldsymbol{v})_D,
\end{aligned}
\end{equation*}
and their corresponding norms are defined as
\begin{equation*}
\|\boldsymbol{u}\|_{H(\text{curl}^2;D)} = \sqrt{(\boldsymbol{u},\boldsymbol{v})_{H(\text{curl}^2;D)}}, \quad \|\boldsymbol{u}\|_{H(\text{div};D)} = \sqrt{(\boldsymbol{u},\boldsymbol{v})_{H(\text{div};D)}}.
\end{equation*}

To impose the boundary conditions and divergence-free condition, we introduce the following subspaces
\begin{equation*}
\begin{aligned}
	H_0(\text{curl}^2;D) &:= \left\{\boldsymbol{u} \in H(\text{curl}^2;D): \boldsymbol{n}\times\boldsymbol{u}=0 \text{ and } \nabla\times \boldsymbol{u}=0 \text{ on } \partial D \right\}, \\
	H(\text{div}^0;D) &:= \left\{\boldsymbol{u} \in H(\text{div};D): \nabla\cdot\boldsymbol{u}=0 \text{ in } D \right\}.
\end{aligned}
\end{equation*}

Throughout this paper, we denote by $C$ a generic positive constant whose value may vary in different occurrences but remains independent of both the mesh size and the location of the interface.

\subsection{The quad-curl interface problem} \label{ssec:interproblem}

Consider a $C^2$-smooth curve $\Gamma$ that divides the domain $\Omega$ into two disjoint subdomains $\Omega_+$ and $\Omega_-$, as depicted in Fig. \ref{fig:domain}. Let $\boldsymbol{n}=(n_1,n_2)^{\mathrm T}$ denote the unit normal vector to $\Gamma$ pointing from $\Omega_-$ to $\Omega_+$.

\begin{figure}[h]
	\begin{center}
	\includegraphics[width=0.35\textwidth]{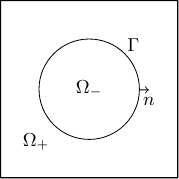}
	\end{center}
	\caption{An illustration of domain $\Omega$ with a circular interface $\Gamma$ and subdomains $\Omega_-$, $\Omega_+$.}
	\label{fig:domain}
\end{figure}

For the operators $\text{d}=\text{div}$ and $\text{d}=\curl\curl$, we define the following piecewise Sobolev spaces
\begin{equation*}
\begin{aligned}
\bm{H}(\text{d};\Omega_-\cup\Omega_+)&:=\left\{\boldsymbol{v}\in\boldsymbol{L}^2(\Omega): \boldsymbol{v}|_{\Omega_j}\in\bm{H}(\text{d},\Omega_j),\ j=\pm\right\},\\
\boldsymbol{H}^s(\Omega_-\cup\Omega_+)&:=\left\{\boldsymbol{v}\in\boldsymbol{L}^2(\Omega): \boldsymbol{v}|_{\Omega_j}\in\boldsymbol{H}^s(\Omega_j),\ j=\pm\right\}.
\end{aligned}
\end{equation*}
The norm on $\bm{H}(\text{d};\Omega_-\cup\Omega_+)$ is defined by
\begin{align*}
	\|\bm v\|^2_{\bm{H}(\text{d};\Omega_-\cup\Omega_+)}:=\left(\|\bm v\|^2_{\bm{H}(\text{d},\Omega_{-})}+\|\bm v\|^2_{\bm{H}(\text{d},\Omega_{+})}\right)^{1/2},
\end{align*}
Moreover, we define the space $H(\text{div}^0;\Omega_-\cup\Omega_+)\subset H(\text{div};\Omega_-\cup\Omega_+)$ to be the set of piecewise divergence-free functions, and $H_0(\text{curl}^2;\Omega_-\cup\Omega_+)\subset H(\text{curl}^2;\Omega_-\cup\Omega_+)$ to be the set of functions with vanishing trace on $\partial\Omega$, i.e., $\boldsymbol{n}\times\boldsymbol{v}=\nabla\times\boldsymbol{v}=0$ on $\partial\Omega$ for $\bm v\in H_0(\text{curl}^2;\Omega_-\cup\Omega_+)$.

Let us now consider the quad-curl interface problem described by the following set of equations
\begin{subequations}\label{OriginProblem}
	\begin{alignat}{2}
	\bm{\nabla}\times\nabla\times\left(\alpha\bm{\nabla}\times\nabla\times \boldsymbol{u}\right) +\gamma \boldsymbol{u} &= \boldsymbol{f}&\quad &\text{in}\ \ \Omega_-\cup\Omega_+,\label{PDE}\\
	\nabla\cdot \bm u &=0&\quad&\text{in}\ \ \Omega_-\cup\Omega_+,\\
	\nabla\times \boldsymbol{u}=\boldsymbol{n}\times \boldsymbol{u} &=0&\quad &\text{on}\ \ \partial\Omega,\label{boundaryCon}\\
	[\![\boldsymbol{n}\times \boldsymbol{u}]\!]&=0&\quad&\text{on}\ \ \Gamma,\label{jumpcon1}\\
	[\![\nabla\times \boldsymbol{u}]\!] &=0&\quad &\text{on}\ \ \Gamma,\label{jumpcon2}\\
	[\![\boldsymbol{n}\times(\alpha\bm{\nabla}\times\nabla\times \boldsymbol{u})]\!]&=\varphi_3&\quad &\text{on}\ \ \Gamma,\label{jumpcon3}\\
	[\![\nabla\times(\alpha\bm{\nabla}\times\nabla\times \boldsymbol{u})]\!]&=\varphi_4&\quad &\text{on}\ \ \Gamma,\label{jumpcon4}\\
	[\![\boldsymbol{n}\cdot \boldsymbol{u}]\!]&=0&\quad&\text{on}\ \ \Gamma.\label{jumpcon5}
	\end{alignat}
\end{subequations}

In this problem, the coefficient $\alpha$ is a piecewise constant function in $\Omega$, specifically $\alpha = \alpha_i>0$ in $\Omega_i$ for $i =\pm$. The constant $\gamma\geq 0$. The function $\boldsymbol{f}\in {H}(\text{div}^0;\Omega_-\cup\Omega_+)$ is given, as well as the functions $\varphi_3\in H^{-1/2}(\Gamma)$ and $\varphi_4\in H^{1/2}(\Gamma)$. The notation $[\![\cdot]\!]_\Gamma$ indicates the jump across $\Gamma$, that is, $[\![v]\!]=v^-|_{\Gamma}-v^+|_{\Gamma}$. In the following, we abbreviate \( v|_{\Omega_-} \) and \( v|_{\Omega_+} \) as \( v^- \) and \( v^+ \), respectively.

\subsection{Weak formulation}\label{ssec:weakform}
To derive the weak formulation for the quad-curl interface problem \eqref{PDE}-\eqref{jumpcon5}, we begin by multiplying both sides of equation \eqref{PDE} by a test function $\boldsymbol{v}\in H_0(\text{curl}^2;\Omega)$ and then integrate by parts twice.
\begin{equation}\label{integrationbyparts}
\begin{aligned}
&\int_{\Omega_-\cup\Omega_+}\left(\alpha\bm{\nabla}\times\nabla\times \boldsymbol{u}\right)\cdot\left(\bm{\nabla}\times\nabla\times \boldsymbol{v}\right)\mathrm{d} A + \int_{\Omega_-\cup\Omega_+}\gamma\boldsymbol{u}\cdot\boldsymbol{v}\mathrm{d} A\\
&+\int_{\Gamma}\boldsymbol{n}\times\left(\alpha\bm{\nabla}\times\nabla\times\boldsymbol{u}\right)^-\cdot\left(\nabla\times\boldsymbol{v}\right)^-\mathrm{d} s - \int_{\Gamma}\nabla\times\left(\alpha\bm{\nabla}\times\nabla\times\boldsymbol{u}\right)^-\cdot\left(\boldsymbol{n}\times\boldsymbol{v} \right)^-\mathrm{d} s\\
&-\int_{\Gamma}\boldsymbol{n}\times\left(\alpha\bm{\nabla}\times\nabla\times\boldsymbol{u}\right)^+\cdot\left(\nabla\times\boldsymbol{v}\right)^+\mathrm{d} s + \int_{\Gamma}\nabla\times\left(\alpha\bm{\nabla}\times\nabla\times\boldsymbol{u}\right)^+\cdot\left(\boldsymbol{n}\times \boldsymbol{v} \right)^+\mathrm{d} s\\
&=\int_{\Omega_-\cup\Omega_+}\boldsymbol{f}\cdot\boldsymbol{v}\mathrm{d} A.
\end{aligned}
\end{equation}

Denote by $\{\!\!\{\xi\}\!\!\}=\kappa_1\xi^-+\kappa_2\xi^+$ the average of $\xi$ on $\Gamma$, where $\kappa_1$ and $\kappa_2$ are constants with $\kappa_1$ + $\kappa_2 = 1$. Recalling the definition of $[\![\xi]\!]$, $\xi^-$ and $\xi^+$ can be expressed in terms of the averages and jumps
\begin{equation}\label{jump_ave}
\xi^-=\{\!\!\{\xi\}\!\!\}+\kappa_2[\![\xi]\!],\quad \xi^+=\{\!\!\{\xi\}\!\!\}-\kappa_1[\![\xi]\!].
\end{equation}
By substituting the expressions from \eqref{jump_ave} into \eqref{integrationbyparts}, we establish a weak formulation of (\ref{OriginProblem}). This formulation seeks $\boldsymbol{u}\in \bm W := H_0(\text{curl}^2;\Omega)\cap H(\text{div}^0;\Omega)$ such that  
\begin{equation}\label{Weak}
a(\boldsymbol{u},\boldsymbol{v})=\mathcal{L}(\boldsymbol{v}),\ \ \text{for all } \boldsymbol{v}\in \bm W,
\end{equation}
where the bilinear form $a:\bm W\times\bm W\to \mathbb{R}$ and the linear form $\mathcal{L}\in \bm W'$ are defined as follows
\begin{equation*}
\begin{aligned}
a(\boldsymbol{u},\boldsymbol{v})&:=\int_{\Omega_{-}\cup\Omega_{+}}\left(\alpha\bm{\nabla}\times\nabla\times \boldsymbol{u}\right)\cdot\left(\bm{\nabla}\times\nabla\times \boldsymbol{v}\right)\mathrm{d} A+\int_{\Omega_{-}\cup\Omega_{+}}\gamma\boldsymbol{u}\cdot\boldsymbol{v} \mathrm{d} A,\\
\mathcal{L}(\boldsymbol{v})&:=\int_{\Omega_-\cup\Omega_{+}}\boldsymbol{f}\cdot\boldsymbol{v}\mathrm{d} A-\int_{\Gamma}\varphi_3\nabla\times \bv\mathrm{d} s+\int_{\Gamma}\varphi_4\boldsymbol{n}\times \bv\mathrm{d} s.
\end{aligned}
\end{equation*}

\begin{theorem}
	There exists a unique solution $\bu\in \bm W$ to the weak formulation \eqref{Weak}.
\end{theorem}
\begin{proof}
	For any $\bu \in \bm W$,  $\nabla\times \bu \in L^2(\Omega)$ and $\nabla\times\nabla\times \bu =((\nabla\times\bu)_y,-(\nabla\times \bu)_x)^{\mathrm T}\in \bm L^2(\Omega)$, from which we can infer that $\nabla\times\bu \in H^1_0(\Omega)$ together with the boundary condition. By applying the Friedrichs inequality, we deduce that $\|\bm u\|_{L^2(\Omega)}\leq\|\nabla\times \bu\|_{L^2(\Omega)}\leq\|\nabla\times\nabla\times \bu\|_{L^2(\Omega)}$, which further leads to the ellipticity of the bilinear form $a(\cdot,\cdot)$ defined on $\bm W$. Additionally, by utilizing the trace inequalities, we easily establish the continuity of both the bilinear form $a(\cdot,\cdot)$ and the linear operator $\mathcal{L}$. Thus, the Lax-Milgram lemma guarantees the existence of a unique solution $\bu$ in $\bm W$.
\end{proof}
\begin{remark}
	In the proof, the crucial aspect lies in estimating the coercivity of the bilinear form $a(\cdot,\cdot)$. Notably, the jump conditions $[\![\bn\times\bu]\!] = [\![\nabla\times\bu]\!]=0$ imply that $\bu\in H(\text{curl}^2;\Omega)$. When combined with the boundary conditions, this result can be readily derived from the Friedrichs inequality. Furthermore, the jump condition $[\![\bn\cdot\bu]\!]=0$ and $\bm u$ belonging to $H(\text{div}^0;\Omega_-,\Omega_{+})$ imply that the solution $\bu$ resides in $H(\text{div}^0;\Omega)$.
\end{remark}

\begin{remark}
	For a sufficiently small value of $\gamma$(tending to zero), the jump condition $[\![\bu\cdot\bn]\!]=0$ on $\Gamma$ plays a crucial role, indicating that $\bu$ belongs to $H(\text{div}^0;\Omega)$. By combining this jump condition with $[\![\bn\times\bu]\!]=[\![\nabla\times\bu]\!]=0$, and utilizing the Friedrichs inequality (See Corollary 3.51 of \cite{Monk2003finite}), we can establish $\|\bu\|_{L^2(\Omega)}\leq\|\nabla\times\bu\|_{L^2(\Omega)}\leq\|\nabla\times\nabla\times\bu\|_{L^2(\Omega)}$. Consequently, the ellipticity of the bilinear form $a(\cdot,\cdot)$ becomes evident.
\end{remark}

\begin{remark}
	The jump conditions \eqref{jumpcon1} and \eqref{jumpcon5} yield $\nabla\times\bu\in L^2(\Omega)$ and $\nabla\cdot\bu\in L^2(\Omega)$. By combining these with the boundary condition \eqref{boundaryCon}, we deduce that $\bu\in \bm H^1(\Omega)$ (see Corollary 2.15 of \cite{amrouche1998vector}). Similarly, the jump condition \eqref{jumpcon2} implies that $\nabla\times\bu \in H^1(\Omega)$. 
\end{remark}

\section{Discrete weak formulation}\label{Sec:discrete}
Let $\mathcal{T}_h$ denote a quasi-uniform, shape-regular triangulation of domain $\Omega$, which is generated independently of the existence of the interface $\Gamma$ (see Fig. \ref{Diagram2}). For $K\in \mathcal{T}_h$, we denote by $h_K$ its diameter, and $h$ the maximum diameter among all elements, i.e., $h=\max_{K\in\mathcal{T}_h}h_K$. The set of all elements intersected by $\Gamma$ is denoted by $\mathcal{T}_h^\Gamma$ and given by
\begin{equation*}
\mathcal{T}_h^\Gamma=\left\{K\in\mathcal{T}_h:\ K\cap\Gamma\neq\emptyset\right\} .
\end{equation*}
For each interface element $K\in\mathcal{T}_h^\Gamma$, $\Gamma_K$ represents the part of $\Gamma$ contained within $K$. We further define $\Omega_{h,j}=\bigcup_{K\in\mathcal{T}_{h,j}}K$, where
\begin{equation*}
\mathcal{T}_{h,j}=\left\{K\in \mathcal T_h:K\cap \Omega_j\neq \emptyset\right\},\quad j=\pm.
\end{equation*}
Two sets of element edges are introduced as follows
\begin{align*}
\mathcal{E}_h&=\{E: E \text{ is an edge of } K \text{ for all } K\in \mathcal T_h \text{ and } E\not\subset \partial\Omega\},\\
\mathcal{E}_h^\Gamma &=\{E: E \text{ is edge of } K \text{ for all } K\in \mathcal{T}_h^\Gamma \text{ and } E\not\subset\partial\Omega\}.
\end{align*}

In addition to the definitions provided earlier, we make specific assumptions regarding the triangulation $\mathcal{T}_h$ of domain $\Omega$ (cf. \cite{Hansbo2001, Massing}). 

\begin{assumption}
We assume that the interface $\Gamma$ intersects each element boundary $\partial K$ precisely twice, and each (open) edge is intersected by $\Gamma$ at most once.
\end{assumption}
\begin{assumption}\label{Assump3}
We assume that a finite number of edges is adequate to traverse continuously from an interface element to a non-interface element
\end{assumption}

%
\begin{figure}[h]
	\centering
	\subfigure{
		\begin{minipage}[t]{0.41\linewidth}
			\centering
			\includegraphics[width=1.8in]{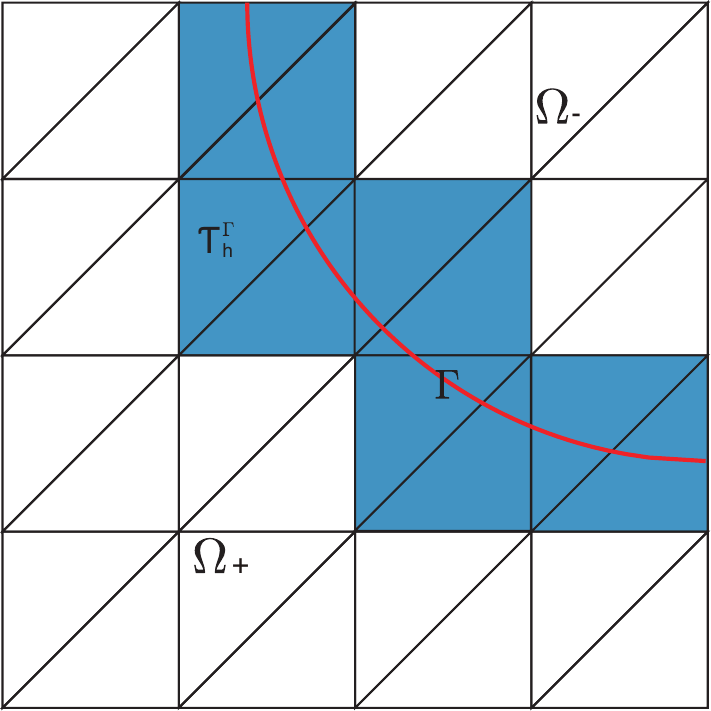}
		\end{minipage}
	}
	\subfigure{
		\begin{minipage}[t]{0.41\linewidth}
			\centering
			\includegraphics[width=1.8in]{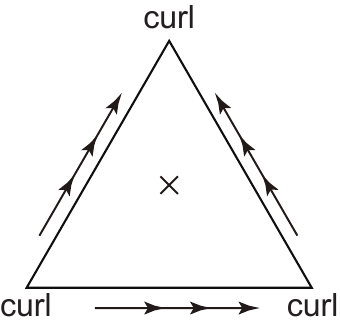}
		\end{minipage}
	}
	\caption{The diagram of local mesh (left) and diagram of degrees of freedom (right).}
	\label{Diagram2}
\end{figure}

We employ the lowest order curl-curl conforming finite element space including $P_2$ polynomials \cite{HuZhangZhang2020}, characterized by the following shape function space for each element $K\in\mathcal{T}_h$
\begin{equation}
\begin{aligned}
\mathcal{S}_h(K) = [P_2(K)]^2 \oplus \text{span}\{\mathfrak{p} B\}.
\end{aligned}
\end{equation}
Here $[P_k(K)]^2$ denotes the space of vector-valued polynomials with a maximum degree of $k$, and $B=\lambda_1\lambda_2\lambda_3$ with barycentric coordinates $\lambda_i$. Additionally, $\mathfrak{p}$ stands for the Poincar\'e operator\cite{Arnold2018}, which maps a scalar function to a vector field 
\begin{equation*}
\mathfrak{p}u:=\int_{0}^{1}t\boldsymbol{x}^\bot u(t\boldsymbol{x})\d t \quad \text{with} \quad \boldsymbol{x}:=(x_1,x_2)^{\mathrm T} \quad \text{and} \quad \boldsymbol{x}^{\bot}:=(-x_2,x_1)^{\mathrm T}.
\end{equation*}
It is worth noting that the operator $\mathfrak{p}$ satisfies $\nabla\times\mathfrak{p}u=u.$

For any $\bm u\in \mathcal{S}_h(K)$, the degrees of freedom can be described as follows (see Fig. \ref{Diagram2})
\begin{itemize}
	\item[-] Vertex values of $\bm u$ at all vertices $v_i$ of element $K$
	\begin{equation}\label{DOF1}
	\begin{aligned}
	\boldsymbol{M}_v(\bm u)=\{(\nabla\times\bm u)(v_i),\ i=1,2,3\}.
	\end{aligned}  
	\end{equation}
	
	\item[-] Edge moments of $\bm u$ on all edges $e_i$ of $K$, each with the unit tangential vector $\boldsymbol{\tau}_i$
	\begin{align}
	\boldsymbol{M}_e(\bm u)=&\left\{\int_{e_i}\bm u\cdot\boldsymbol{\tau_i} q ds\ \forall q\in P_2(e_i),\ i=1,2,3\right\}.
	\end{align}  
	
	\item[-] Interior degrees of freedom of $\bm u$ in $K$
	\begin{align}\label{DOF3}
	\boldsymbol{M}_K(\bm u)=\left\{\int_{K}\bm u\cdot\boldsymbol{q} dx\ \forall \boldsymbol{q}\in \{\boldsymbol{x}\}\right\}.
	\end{align}
\end{itemize}

By assembling the shape functions on neighboring elements using the aforementioned degrees of freedom, we construct a global finite element space, denoted by $\mathcal{S}_h(\mathcal{T}_h)$.

To establish the finite element space for the quad-curl interface problem \eqref{PDE}-\eqref{jumpcon5}, we first introduce the curl-curl conforming finite element space $\mathcal S_h(\mathcal{T}_{h,j})$ on $\mathcal{T}_{h,j}$ for $j=\pm$. Then, we define the unfitted finite element space as follows
\begin{equation*}
\boldsymbol{V}_h=\big\{\bm v_h=\{\bm v_h^-,\bm v_h^+\}:\ \boldsymbol{v}_h^j\in\mathcal S_h(\mathcal T_{h,j}) \text{ for } j=\pm,\text{ and } \bn\times\bm v_h^+=\nabla\times\bm v_h^+=0\text{ on } \partial\Omega \big\}.
\end{equation*}

The unfitted finite element scheme based on Nitsche's method  for the quad-curl problem (\ref{Weak}) is formulated to find $\bu_h\in \boldsymbol{V}_h$ such that 
\begin{equation}\label{DisWeak}
\mathcal{A}_h(\bu_h,\bv_h):=a_h(\bu_h,\bv_h)+J_1(\bm u_h,\bm v_h)+J_2(\bm u_h,\bm v_h)=\mathcal{L}_h(\bv_h)\quad \forall \bv_h\in \boldsymbol{V}_h.
\end{equation}
In the above equation, the discrete bilinear form $a_h$ and $J$ are defined as follows
\begin{equation*}
\begin{aligned}
a_h(\bu_h,\bv_h)&:=\int_{\Omega_-\cup\Omega_+}\left(\alpha\bm{\nabla}\times\nabla\times \boldsymbol{u}_h\right)\cdot\left(\bm{\nabla}\times\nabla\times \boldsymbol{v}_h\right)\d A+\int_{\Omega_-\cup\Omega_+}\gamma\boldsymbol{u}_h\cdot\boldsymbol{v}_h \d A\\
&\quad+\int_{\Gamma}\{\!\!\{\boldsymbol{n}\times(\alpha\bm{\nabla}\times\nabla\times\boldsymbol{u}_h)\}\!\!\}[\![\nabla\times\boldsymbol{v}_h]\!]+\{\!\!\{\boldsymbol{n}\times(\alpha\bm{\nabla}\times\nabla\times\boldsymbol{v}_h)\}\!\!\}[\![\nabla\times\boldsymbol{u}_h]\!]\d s\\
&\quad-\int_{\Gamma}\{\!\!\{\nabla\times(\alpha\bm{\nabla}\times\nabla\times\boldsymbol{u}_h)\}\!\!\}[\![\bn\times\bv_h]\!]+\{\!\!\{\nabla\times(\alpha\bm{\nabla}\times\nabla\times\boldsymbol{v}_h)\}\!\!\}[\![\bn\times\bu_h]\!]\d s\\
&\quad+h^{-2}\int_{\Omega_-\cup\Omega_+}(\nabla_h \cdot\bu_h) (\nabla_h \cdot \bv_h) \d A +\sum_{i=0}^{2}\mathcal{G}_i(\bu_h,\bv_h).
\end{aligned}
\end{equation*}
\begin{equation*}
\begin{aligned}
\mathcal{G}_0(\bu_h,\bv_h)&:=\sum_{K\in \mathcal T_h^{\Gamma}}h^{-3}\int_{\Gamma_K}[\![\bn\cdot\bu_h]\!][\![\bn\cdot\bv_h]\!]\d s+\sum_{E\in\mathcal{E}_h}h^{-3}\int_{E}[\![\bn\cdot\bu_h]\!][\![\bn\cdot\bv_h]\!]\d s,\\
\mathcal{G}_1(\bu_h,\bv_h)&:=\lambda\sum_{K\in\mathcal T_h^{\Gamma}}h^{-3}\int_{\Gamma_K}[\![\bn\times \bu_h]\!][\![\bn\times\bv_h]\!]\d s,\\
\mathcal{G}_2(\bu_h,\bv_h)&:=\lambda\sum_{K\in\mathcal T_h^{\Gamma}}h^{-1}\int_{\Gamma_K}[\![\nabla\times\bu_h]\!][\![\nabla\times\bv_h]\!]\d s,\\
J_1(\bm u_h,\bm v_h)&:=\sum_{E\in\mathcal{E}_h^{\Gamma}}\sum_{0\leq l\leq 4}h^{2l-1}\int_{E}[\![\partial_n^l\bm u_h]\!]\cdot[\![\partial_n^l\bm v_h]\!]\d s,\\
J_2(\bm u_h,\bm v_h)&:=\sum_{E\in\mathcal{E}_h^{\Gamma}}\sum_{0\leq l\leq 4}h^{2l-1}\int_{E}[\![\partial_n^l(\nabla\times\bm u_h)]\!]\cdot[\![\partial_n^l(\nabla\times\bm v_h)]\!]\d s,
\end{aligned}
\end{equation*}
where $\nabla_h\cdot$ denotes taking divergence piecewise, and $\partial_n^lq=\sum_{|m|=l}D^mq(\bm x)\bn^m$ for multi-index $m=(m_1,m_2)$ with $|m|=m_1+m_2$, and $\bn^m=n_1^{m_1}n_2^{m_2}$. The discrete linear form $\mathcal{L}_h$ is defined as follows
\begin{align*}
\mathcal{L}_h(\boldsymbol{v}_h)&:=\int_{\Omega_-\cup\Omega_+}\boldsymbol{f}\cdot\boldsymbol{v}_h\d A-\int_{\Gamma}\varphi_3\{\!\!\{\nabla\times \boldsymbol{v}_h\}\!\!\}^*\d s+\int_{\Gamma}\varphi_4\{\!\!\{ \boldsymbol{n}\times\boldsymbol{v}_h\}\!\!\}^*\d s.
\end{align*}
where $\{\!\!\{\xi\}\!\!\}^*=\kappa_2\xi^-+\kappa_1\xi^+$.

\begin{remark}
The interface terms $\mathcal G_1$ and $\mathcal G_2$ are introduced to enforce the interface conditions and the stability of the numerical scheme. The terms $\{\!\!\{\boldsymbol{n}\times(\alpha\bm{\nabla}\times\nabla\times\boldsymbol{v}_h)\}\!\!\}[\![\nabla\times\boldsymbol{u}_h]\!]$ and $\{\!\!\{\nabla\times(\alpha\bm{\nabla}\times\nabla\times\boldsymbol{v}_h)\}\!\!\}[\![\bn\times\bu_h]\!]$ contribute to the symmetry of the numerical scheme.
\end{remark}

\begin{remark}
When the parameter $\gamma >0$, the term $\int_{\Omega_-\cup\Omega_+}\gamma\boldsymbol{u}_h\cdot\boldsymbol{v}_h \d A$ theoretically ensures the divergence-free condition. However, all the interface terms have significantly larger magnitudes compared to $\int_{\Omega_-\cup\Omega_+}\gamma\boldsymbol{u}_h\cdot\boldsymbol{v}_h \d A$, which may compromise the numerical scheme's stability. To address this, the terms $h^{-2}\int_{\Omega_-\cup\Omega_+}(\nabla_h\cdot \bu_h) (\nabla_h\cdot \bv_h) \d A$ and $\mathcal G_0$ are introduced to enforce the divergence-free condition (see Section \ref{Sec:error}) and enhance the scheme's stability. These additional terms ensure that the scheme remains stable even when $\gamma$ tends to zero, as discussed in Section \ref{Sec:wellposedness}. Moreover, these terms also help reduce the condition number of the scheme (see Section \ref{Sec:conditionnumber}).
\end{remark}

\begin{remark}
The Ghost penalty $J$ is incorporated to address elements that the intersection with $\Omega$ has a significantly smaller measure $|K\cap\Omega|$ compared to the element size $|K|$ in the vicinity of the interface $\Gamma$ (see, e.g., \cite{Massing, BadiaNe2022}). This term improves the stability and accuracy of the numerical scheme.
\end{remark}

\section{The existence and uniqueness of discrete scheme}\label{Sec:wellposedness}
In this section, we will establish the existence, uniqueness, and stability of the discrete scheme \eqref{DisWeak} using the Lax-Milgram lemma. To do so, we need to prove the coercivity and continuity of the bilinear form $\mathcal{A}_h$ and the continuity of the linear functional $\mathcal{L}_h$ with respect to the discrete norm $\3bar \cdot\3bar_{\ast}$ defined as
\begin{equation}\label{equ_discretenorm}
\begin{aligned}
\3bar \bv\3bar^2_{\ast} :=&\|\bm{\nabla}\times\nabla\times \bv\|^2_{\Omega_{h,-}\cup\Omega_{h,+}}+\|\bv\|^2_{\Omega_{h,-}\cup\Omega_{h,+}}+h^{-2}\|\nabla_h \cdot\bv\|^2_{\Omega_-\cup\Omega_+}	\\&+\sum_{i=0}^{2}\mathcal{G}_{i}(\bv,\bv)+\sum_{i=1}^2J_i(\bm v,\bm v).
\end{aligned}
\end{equation}

We begin by presenting a crucial trace inequality that involves the interface $\Gamma$ (see \cite{Hansbo2001}). This inequality plays a critical role in the analysis of well-posedness and the subsequent investigation of errors.
\begin{lemma}
[{\cite[Lemma 3]{Hansbo2001}}]\label{traceineq} Under the assumptions of the mesh, the following trace inequalities hold
\begin{equation}\label{TraceResult} 
\begin{aligned}
\|w\|_{\Gamma_K}&\leq Ch^{-1/2}\|w\|_{K},\quad \forall w\in P_k(K) \text{ and }K\in\mathcal T_h^{\Gamma},\\
\|w\|^2_{\Gamma_K}&\leq C\left(h^{-1}\|w\|^2_{K}+h\|w\|^2_{1,K}\right),\quad \forall w\in H^1(K) \text{ and }K\in\mathcal T_h^{\Gamma},
\end{aligned}
\end{equation}
where the constant $C>0$ is independent of the position of $\Gamma$ with respect to the mesh.
\end{lemma}

In addition to the trace inequalities in Lemma \ref{traceineq}, we will also utilize the following trace inequalities and inverse inequality\cite{Ci2002, BrennerScott2008}
\begin{align} 
&\|w\|_{E}\leq Ch^{-1/2}\|w\|_{K},\quad \forall w\in P_k(K), \ E\subset \partial K, \text{ and }K\in\mathcal T_h,\label{traceinq1}\\
&\|w\|^2_{E}\leq C\left(h^{-1}\|w\|^2_{K}+h\|w\|^2_{1,K}\right),\quad \forall w\in H^1(K), \ E\subset \partial K, \text{ and }K\in\mathcal T_h,\label{traceinq2}\\
&\|\nabla w\|_K\leq Ch^{-1}\|w\|_K,\quad \forall w\in P_k(K) \text{ and }K\in\mathcal T_h.\label{inverseinq}
\end{align}

Using the aforementioned inequalities, we are able to demonstrate the extended stability with ghost penalty terms
\begin{lemma}\label{ghostestimate}
Let $\Omega_-$, $\Omega_+$, $\Omega_{h,-}$, $\Omega_{h,+}$, and $\mathcal{E}_h^{\Gamma}$ be defined as in Section \ref{Sec:discrete}. Then, for $\bm v_h=\{\bm v_h^-,\bm v_h^+\}$, where $\bm v_h^i$ is a vector-valued piecewise polynomial on $\Omega_{h,i}$ for $i=\pm$, the following estimates hold
\begin{align}
&\|\bv_h\|^2_{\Omega_{h,-}\cup\Omega_{h,+}}\leq C\Big(\|\bv_h\|^2_{\Omega_-\cup\Omega_+}+h^2J_1(\bm v_{h}, \bm v_{h})\Big),\label{ghostu}\\
&\|\nabla\times\bv_h\|^2_{\Omega_{h,-}\cup\Omega_{h,+}}\leq C\Big(\|\nabla\times\bv_h\|^2_{\Omega_-\cup\Omega_+}+h^2J_2(\bm v_h,\bm v_h)\Big),\label{ghostcurlu}\\
&\|\bm{\nabla}\times\nabla\times\bv_h\|^2_{\Omega_{h,-}\cup\Omega_{h,+}}\leq C\Big(\|\bm{\nabla}\times\nabla\times\bv_h\|^2_{\Omega_-\cup\Omega_+}+J_2(\bm v_h,\bm v_h)\Big)\label{ghostcurlcurlu}.
\end{align}
\end{lemma}

\begin{proof}
The proof for \eqref{ghostu} and \eqref{ghostcurlu} closely mirrors the approach in \cite[Proposition 5.1]{Massing}, and is therefore omitted here. To derive \eqref{ghostcurlcurlu}, for any interface element $K_1\in \mathcal{T}_h^{\Gamma}$, there exist a non-interface element $K_2$ such that
\begin{align*}
\|\bm{\nabla}\times\nabla\times \bv_h\|^2_{K_1}&=|\bm{\nabla}\times(\nabla\times \bv_h-c_0)\|^2_{K_1}\leq Ch^{-2}\|\nabla\times \bv_h -c_0\|^2_{K_1}\\
&\leq C\Big(h^{-2}\|\nabla\times \bv_h -c_0\|^2_{K_2}+J_2(\bm v_h,\bm v_h)\Big)\\
&\leq C\Big(\|\bm{\nabla}\times\nabla\times\bv_h\|^2_{K_2}+J_2(\bm v_h,\bm v_h)\Big),
\end{align*}
where $c_0=|K_2|^{-1}\int_{K_2}\nabla\times\bv_h \d A$ is a constant and we have used \eqref{ghostcurlu} in the second inequality. We then proceed as in the proof of \eqref{ghostu} and \eqref{ghostcurlu} to establish \eqref{ghostcurlcurlu}.
\end{proof}

Furthermore, we establish the following Poincar\' e-type inequalities.
\begin{lemma}\label{poincare}
For $\bm v_h\in \bm V_h$, the following Poincar\' e-type inequalities hold
\begin{align}
&\|\nabla\times\bm v_h\|_{\Omega_{-}\cup\Omega_{+}}^2\leq C\Big(\|\bm{\nabla}\times\nabla\times\bm v_h\|_{\Omega_{-}\cup\Omega_{+}}^2+\mathcal G_2(\bm u_h,\bm u_h)\Big),\label{curlv-poincare}\\
&\|\bm v_h\|^2\leq C\Big(\|\nabla\times\bm v_h\|_{\Omega_{-}\cup\Omega_{+}}^2+ \|\nabla_h\cdot\bm v_h\|_{\Omega_{-}\cup\Omega_{+}}^2+\mathcal G_0(\bm u_h,\bm u_h)+\mathcal G_1(\bm u_h,\bm u_h)\Big).\label{v-poincare}
\end{align}
\end{lemma}
\begin{proof}
Recalling the definition of $\bm V_h$, we can observe that $\nabla\times(\bm v_h|_{\Omega_{i}})$ 
belong to $L^2(\Omega_{i})$ for $i = \pm$. We can use $\theta_h$ to denote a function in the $L^2(\Omega)$ space, such that $\theta_h|_{\Omega_i} = \nabla\times(\bm v_h|_{\Omega_{i}})$.  Define $\psi$ by
\begin{equation}\label{psidef}
\begin{aligned}
-\Delta \psi &= \theta_h\quad &&\text{in }\Omega,\\
\psi&=0\quad &&\text{on }\partial\Omega.
\end{aligned}
\end{equation}
Then, according to the regularity of the Laplace problem, we have
\begin{align}\label{reg}
\|\psi\|^2_2\leq C\|\theta_h\|^2 = C\|\nabla\times\bm v_h\|^2_{\Omega_{-}\cup\Omega_{+}}.
\end{align}
Multiplying the first equation of \eqref{psidef} by $\nabla\times\dot{\bm v}_h$ and integrating over $\Omega$ by parts, we have
\begin{align*}
&\|\theta_h\|^2=(-\Delta \psi,\theta_h)\\
=&(\nabla\psi,\nabla\nabla\times\bm v_h)_{\Omega_{-}\cup\Omega_{+}}-\sum_{K\in\mathcal T_h^{\Gamma}}\int_{\Gamma_K}[\![\nabla\times\bm v_h]\!]\nabla\psi \cdot \bm n \d s\\
\leq &\Big(\|\bm{\nabla}\times\nabla\times\bm v_h\|_{\Omega_{-}\cup\Omega_{+}}^2+\mathcal G_2(\bm v_h,\bm v_h)\Big)^{1/2}\Big(\|\nabla\psi\|^2+\sum_{K\in\mathcal T_h^{\Gamma}}h{\lambda}^{-1}\|\bm n\cdot\nabla\psi\|^2_{\Gamma_K}\Big)^{1/2},
\end{align*}
where we have used the fact that $\nabla\times\bm v_h$ is continuous in $\Omega_j$ for $j=\pm$, and $\nabla\times\bm v_h=0$ on $\partial\Omega$. 
Furthermore, applying trace inequality \eqref{TraceResult} to $\sum_{K\in\mathcal T_h^{\Gamma}}h{\lambda}^{-1}\|\bm n\cdot\nabla\psi\|^2_{\Gamma_K}$ and using \eqref{reg}, we obtain

\[\|\nabla\psi\|^2+\sum_{K\in\mathcal T_h^{\Gamma}}h{\lambda}^{-1}\|\bm  n\cdot\nabla\psi\|^2_{\Gamma_K}\leq C\|\psi\|_2^2\leq C\|\nabla\times\bm v_h\|_{\Omega_{-}\cup\Omega_{+}}^2,\]
which completes the proof of \eqref{curlv-poincare}.

To prove \eqref{v-poincare}, define $\bm w$ by
\begin{equation}\label{wdef}
\begin{aligned}
\nabla\times\nabla\times \bm w -\nabla\nabla \cdot\bm w &= \bm v_h\quad &&\text{in }\Omega,\\
\bm w\times \bm n &=0\quad &&\text{on }\partial\Omega,\\
\nabla\cdot\bm w&=0\quad &&\text{on }\partial\Omega. 
\end{aligned}
\end{equation}
According to \cite[Theorem 4.8]{Arnold2018}, we have
\begin{align}\label{w-reg}
\|\nabla\cdot\bm w\|_1+\|\nabla\times\bm w\|_1\leq C\|\bm v_h\|.
\end{align}
Multiplying the first equation of \eqref{wdef} by $\bm v_h$ and integrating over $\Omega$ by parts, we have
	\begin{align*}
	\|\bm v_h\|^2&=(\nabla\times\nabla\times \bm w -\nabla\nabla\cdot\bm w,\bm v_h)_{\Omega_{-}\cup\Omega_{+}}\\
	&=(\nabla\times \bm w ,\nabla\times\bm v_h)_{\Omega_{-}\cup\Omega_{+}}+(\nabla\cdot\bm w ,\nabla_h\cdot\bm v_h)_{\Omega_{-}\cup\Omega_{+}}\\
	&+\sum_{K\in\mathcal T_h^{\Gamma}}\int_{\Gamma_K}[\![\bm n\times\bm v_h]\!]\nabla\times\bm w\d s+\sum_{K\in\mathcal T_h^{\Gamma}}\int_{\Gamma_K}[\![\bm n\cdot\bm v_h]\!]\nabla\cdot\bm w\d s\\
	&+\sum_{E\in\mathcal E_h}\int_{E}[\![\bm n\cdot\bm v_h]\!]\nabla\cdot\bm w\d s\\
	\leq & C\Big(\|\nabla\times\bm v_h\|_{\Omega_{-}\cup\Omega_{+}}^2+ \|\nabla_h\cdot\bm v_h\|_{\Omega_{-}\cup\Omega_{+}}^2+\mathcal G_0(\bm v_h,\bm v_h)+\mathcal G_1(\bm v_h,\bm v_h)\Big)^{1/2}\\
	&\Big(\|\nabla\cdot\bm w\|^2_1+\|\nabla\times\bm w\|^2_1\Big)^{1/2},
	\end{align*}
which, together with \eqref{w-reg}, yields \eqref{v-poincare}.
\end{proof}

With the previous preparation, we now establish the coercivity and continuity of the discrete bilinear form $\mathcal{A}_h(\cdot, \cdot)$ with respect to the discrete norm \eqref{equ_discretenorm}.
\begin{lemma}\label{CoreCon}
Suppose that the penalty parameter $\lambda$ is sufficiently large, then the bilinear form $\mathcal{A}_h$ satisfies the following inequalities
\begin{align}
&\left|\mathcal{A}_h(\bu_h,\bv_h)\right|\leq c_2\3bar \bu_h\3bar_{\ast} \3bar\bv_h\3bar_{\ast}, \quad \forall \bu_h,\bv_h\in \boldsymbol{V}_h\label{continuity},\\
&\mathcal{A}_h(\bu_h,\bu_h)\geq c_1\3bar \bu_h\3bar_{\ast}^2,\quad \forall \bu_h\in \boldsymbol{V}_h.\label{coercivity}
\end{align}
\end{lemma}

\begin{proof}
To prove \eqref{continuity}, we first use the trace inequality (\ref{TraceResult}) and inverse inequality \eqref{inverseinq} to obtain
\begin{align*}
&\|\bn\times(\alpha\bm{\nabla}\times\nabla\times\bu_h)^i\|_{\Gamma_K}\leq Ch^{-\frac{1}{2}}\|(\bm{\nabla}\times\nabla\times\bu_h)^i\|_{K}\quad \text{for }i=\pm,\\
&\|\nabla\times(\alpha\bm{\nabla}\times\nabla\times\bu_h)^i\|_{\Gamma_K}\leq Ch^{-\frac{3}{2}}\|(\bm{\nabla}\times\nabla\times\bu_h)^i\|_{K}\quad \text{for }i=\pm.
\end{align*}
Furthermore, recalling that $\{\!\!\{\xi\}\!\!\}=k_1\xi^-+k_2\xi^+$ and applying Young's inequality, we have
\begin{equation}\label{Average2}
\begin{aligned}
&\quad\|\{\!\!\{\bn\times(\alpha\bm{\nabla}\times\nabla\times\bu_h)\}\!\!\}\|^2_{\Gamma_K} \\
&\leq Ch^{-1}\|(\bm{\nabla}\times\nabla\times\bu_h)^-\|^2_{K}+Ch^{-1}\|(\bm{\nabla}\times\nabla\times\bu_h)^+\|^2_{K},\\
&\quad\|\{\!\!\{\nabla\times(\alpha\bm{\nabla}\times\nabla\times\bu_h)\}\!\!\}\|^2_{\Gamma_K} \\
&\leq Ch^{-3}\|(\bm{\nabla}\times\nabla\times\bu_h)^-\|^2_{K}+Ch^{-3}\|((\bm{\nabla}\times\nabla\times\bu_h)^+\|^2_{K}.
\end{aligned}
\end{equation}
We can then obtain \eqref{continuity} by applying Cauchy-Schwarz inequality and \eqref{Average2}.

To prove \eqref{coercivity}, we first have
\begin{equation}\label{CoBilinear}
\begin{aligned}
\mathcal{A}_h(\bu_h,\bu_h)&=\3bar \bu_h\3bar^2 +J_1(\bm u_{h},\bm u_{h})+J_2(\bm u_h,\bm u_h) \\
&\quad+2\sum_{K\in \mathcal T_h^{\Gamma}}\int_{\Gamma_K}\{\!\!\{\bn\times(\alpha\bm{\nabla}\times\nabla\times\bu_h)\}\!\!\}\left[\![\nabla\times \bu_h\right]\!]\d s\\
&\quad-2\sum_{K\in\mathcal T_h^{\Gamma}}\int_{\Gamma_K}\{\!\!\{\nabla\times(\alpha\bm{\nabla}\times\nabla\times\bu_h)\}\!\!\}\left[\![\bn\times \bu_h\right]\!]\d s.
\end{aligned}
\end{equation}
where  $\3bar \bv\3bar^2 :=\sum_{j=\pm}\|\alpha^{1/2}\bm{\nabla}\times\nabla\times \bv\|^2_{\Omega_{j}}+\gamma\sum_{j=\pm}\|\bv\|^2_{\Omega_{j}}+\sum_{j=\pm}h^{-2}\|\nabla_h\cdot \bv\|^2_{\Omega_j}+\sum_{i=0}^{2}\mathcal{G}_{i}(\bv,\bv)
$.
According to Lemma \ref{poincare}, we have
\[\3bar \bu_h\3bar^2\geq C\big(\|\bm u_h\|^2_{\Omega_-\cup\Omega_+}+\|\nabla \times\bm u_h\|_{\Omega_-\cup\Omega_+}^2+\3bar \bu_h\3bar^2\big) \text{ when }h\leq 1,
\]
which, together with Lemma \ref{ghostestimate}, leads to
\begin{equation}\label{CoBilinear_omegah}
\begin{aligned}
\mathcal{A}_h(\bu_h,\bu_h)\geq c_0\3bar \bu_h\3bar_{\ast}^2&+2\sum_{K\in \mathcal T_h^{\Gamma}}\int_{\Gamma_K}\{\!\!\{\bn\times(\alpha\bm{\nabla}\times\nabla\times\bu_h)\}\!\!\}\left[\![\nabla\times \bu_h\right]\!]\d s\\
&-2\sum_{K\in\mathcal T_h^{\Gamma}}\int_{\Gamma_K}\{\!\!\{\nabla\times(\alpha\bm{\nabla}\times\nabla\times\bu_h)\}\!\!\}\left[\![\bn\times \bu_h\right]\!]\d s.
\end{aligned}
\end{equation}

For $K\in\mathcal T_h^{\Gamma}$, applying the Cauchy-Schwarz inequality, Young's inequality, and \eqref{Average2}, we have
\begin{align*}
&\left|\int_{\Gamma_K}\{\!\!\{\bn\times(\alpha\bm{\nabla}\times\nabla\times\bu_h)\}\!\!\}\left[\![\nabla\times\bu_h\right]\!] \d s\right| \leq\|\{\!\!\{\bn\times(\alpha\bm{\nabla}\times\nabla\times\bu_h)\}\!\!\}\|_{\Gamma_K}\|\left[\![\nabla\times\bu_h\right]\!]\|_{\Gamma_K}\\
&\leq\frac{h}{C_0\lambda}\|\{\!\!\{\bn\times(\alpha\bm{\nabla}\times\nabla\times\bu_h)\}\!\!\}\|^2_{\Gamma_K}+\frac{c_0\lambda}{4h}\|\left[\![\nabla\times\bu_h\right]\!]\|^2_{\Gamma_K}\\
&\leq\frac{C}{\lambda}\|(\bm{\nabla}\times\nabla\times\bu_h)^-\|^2_{K}+\frac{C}{\lambda}\|(\bm{\nabla}\times\nabla\times\bu_h)^+\|^2_{K}+\frac{c_0\lambda}{4h}\|\left[\![\nabla\times\bu_h\right]\!]\|^2_{\Gamma_K},
\end{align*}
and
\begin{align*}
&\left|\int_{\Gamma_K}\{\!\!\{\nabla\times(\alpha\bm{\nabla}\times\nabla\times\bu_h)\}\!\!\}\left[\![\bn\times\bu_h\right]\!] \d s\right| \\
\leq &\frac{C}{\lambda}\|(\bm{\nabla}\times\nabla\times\bu_h)^-\|^2_{K}+\frac{C}{\lambda}\|(\bm{\nabla}\times\nabla\times\bu_h)^+\|^2_{K}+\frac{c_0\lambda}{4h^3}\|\left[\![\bn\times\bu_h\right]\!]\|^2_{\Gamma_K},
\end{align*}
which, plugged into \eqref{CoBilinear_omegah}, yields \eqref{coercivity} when $\lambda$ is sufficiently large.  
\end{proof}

We are now in a position to prove our main well-posedness result.

\begin{theorem}
The discrete problem (\ref{DisWeak}) has a unique solution when taking the penalty parameter $\lambda$ sufficiently large.
\end{theorem}

\begin{proof} 
Since the discrete problem \eqref{DisWeak} is finite-dimensional, the existence of a solution is equivalent to the uniqueness. It suffices to prove uniqueness. To this end, we consider the associated homogeneous problem and show that $\bu_h=0$ is the only solution. 

By Lemma \ref{CoreCon}, we have
\[0=\mathcal A_h(\bm u_h,\bm u_h)\geq c_1\3bar\bm u_h\3bar_{\ast}^2,\]
which implies $\bm u_h=0$ since $\3bar\cdot\3bar_{\ast}$ is a norm. Hence, the solution to the homogeneous problem is trivial, and this implies the uniqueness of the solution to the original discrete problem. Thus, the discrete problem (\ref{DisWeak}) has a unique solution when the penalty parameter $\lambda$ is sufficiently large.
\end{proof}

\section{Error analysis}\label{Sec:error}
In this section, we will derive a priori error estimates in the discrete norm
\begin{equation}\label{equ:newnorm}
\begin{aligned}
\3bar \bv\3bar_h^2 :=&\|\bm{\nabla}\times\nabla\times \bv\|^2_{\Omega_{-}\cup\Omega_+}+\|\bv\|^2_{\Omega_{-}\cup\Omega_+}+h^{-2}\|\nabla_h\cdot \bv\|^2_{\Omega_{-}\cup\Omega_+}+\sum_{i=0}^{2}\mathcal{G}_{i}(\bv,\bv) \\
&+\sum_{K\in\mathcal T_h^{\Gamma}}\frac{h^3}{\lambda}\|\{\!\!\{\nabla\times\left(\alpha\bm{\nabla}\times\nabla\times\bv\right)\}\!\!\}\|^2_{\Gamma_K}+\sum_{K\in\mathcal T_h^{\Gamma}}\frac{h}{\lambda}\|\{\!\!\{\bn\times\left(\alpha\bm{\nabla}\times\nabla\times\bv\right)\}\!\!\}\|^2_{\Gamma_K}.
\end{aligned}
\end{equation}
Using Lemma \ref{traceineq}, we can show that
\begin{align}
\3bar \bv\3bar_h\leq C\3bar \bv\3bar_{\ast}\quad\text{for }\bm v\in \bm V_h,\label{normrelation}\\
a_h(\bm u_h,\bm v_h)\leq C\3bar\bm u_h\3bar_{h}\3bar\bm v_h\3bar_{h}.\label{errorestimate}
\end{align}

In addition, we will present some Sobolev extension results and approximation properties of $\boldsymbol{V}_h$.
\begin{lemma}[{\cite{Calder1961, Stein1966}}]\label{ExtenLem}
For $i=\pm$, assume $\Omega_i$ is a bounded Lipschitz domain. For $s\geq 0$, there exist extension operators $\mathcal{E}_i: \bm H^s(\Omega_i)\to \bm H^s(\Omega)$ such that $\mathcal{E}_i\bu^i=\bu^i$ in $\Omega_i$ and $\|\mathcal{E}_i\bu^i\|_{s,\Omega}\leq C\|\bu^i\|_{s,\Omega_i}$ for $\bu^i\in \bm H^s(\Omega_i)$. Here, the constant $C=C(s,\Omega_i)$.
\end{lemma}

We present an interpolation error estimate. 
\begin{lemma} [{\cite{HuZhangZhang2020}}]\label{EstimateElem}
Suppose $\bu\in \bH^{s+1}(\Omega)$, $s\geq 1+\delta$ with $\delta>0$, then the following interpolation error estimates hold
\begin{align}
\|\bu-\Pi_h\bu\|_m\leq C h^{\min\{s+1,3\}-m}\|\bu\|_{s+1}\text{ for }m=0,1,2,3,4,
\end{align}
where $\Pi_h : H^{s+1}(\Omega)\to \mathcal S_h(\mathcal T_h)$ is the interpolation operator defined by the degrees of freedom \eqref{DOF1} -- \eqref{DOF3}. 
\end{lemma}

To ease the notation, for $\bu\in \bH^s(\Omega_-\cup\Omega_+)$, denote $\bu^i=\bu|_{\Omega_i}$ and $ \bu^i_{\mathcal E}=\mathcal{E}_i\bu^i\in \bH^s(\Omega)$ for $i=\pm$. Then the interpolation $\Pi_h^{\ast} \bu\in\bm V_h$ is defined as  
\begin{equation}
\Pi_h^{\ast}\bu=\Pi_h\bu^i_{\mathcal E} \text{ in } \Omega_{h,i}\text{ for } i=\pm.
\end{equation}

For the interpolation operator $\Pi_h^{\ast}$, we can prove the following approximation property.
\begin{lemma}\label{InterpolationEstimate}
	Suppose $\bu\in \bH^3(\Omega_-\cup\Omega_+)$, then the following estimate holds
	\begin{equation}
	\3bar \bu-\Pi_h^{\ast}\bu\3bar_h\leq Ch\sum_{i=\pm}\|\bu\|_{3,\Omega_i}.
	\end{equation}
\end{lemma}
\begin{proof}
	Denote $\bxi^i=\bu^i_{\mathcal E}-\Pi_h\bu^i_{\mathcal E}$ for $i=\pm$ and apply Lemma \ref{EstimateElem}, then we have 
	\begin{equation*}
	\|\bxi^i\|+h\|\nabla\times\bxi^i\|+h^2\|\bm{\nabla}\times\nabla\times\bxi^i\|\leq Ch^3\|\bu^i_{\mathcal E}\|_{\bH^3(\Omega)}\leq Ch^3\|\bu^i\|_{3,\Omega_i},\quad i=\pm.
	\end{equation*}
	Denote $\bm \xi=\bu-\Pi_h^{\ast}\bu$, then the fact $\bm \xi=\bm \xi^i$ in $\Omega_i$ yields
	\begin{equation}
	\|\bm{\nabla}\times\nabla\times\bxi\|^2_{\Omega_{-}\cup\Omega_+}+\|\bxi\|^2_{\Omega_{-}\cup\Omega_+}\leq Ch^2\sum_{i=\pm}\|\bu\|^2_{3,\Omega_i}.
	\end{equation}
	We now proceed to estimate the interface terms. Using (\ref{TraceResult}) and Lemma \ref{EstimateElem}, we get
	\begin{equation*}
	\begin{aligned}
	\mathcal{G}_1(\bxi,\bxi)&=\lambda h^{-3}\sum_{K\in\mathcal T_h^{\Gamma}}\|\left[\![\bn\times\bxi\right]\!]\|^2_{\Gamma_K}\leq {C}{h^{-3}}\sum_{K\in\mathcal T_h^{\Gamma}}\sum_{i=\pm}\|\bn\times\bxi^i\|^2_{\Gamma_K}\\
	&\leq {C}{h^{-3}}\sum_{K\in\mathcal T_h^{\Gamma}}\sum_{i=\pm}\left(h^{-1}\|\bxi^i\|^2_{K}+h\|\bxi^i\|^2_{1,K}\right)\\
	&\leq Ch^2\sum_{i=\pm}\|\bu^i\|^2_{3,\Omega_i}.
	\end{aligned}
	\end{equation*}	
	Similarly,
	\begin{equation*}
	\begin{aligned}
	&\mathcal{G}_0(\bxi,\bxi)\leq Ch^2\sum_{i=\pm}\|\bu\|_{3,\Omega_i}^2,\quad\mathcal{G}_2(\bxi,\bxi)\leq Ch^2\sum_{i=\pm}\|\bu\|_{3,\Omega_i}^2,\\
	&\sum_{K\in\mathcal T_h^{\Gamma}}\frac{h^3}{\lambda}\|\{\!\!\{\nabla\times \left(\alpha\bm{\nabla}\times\nabla\times \bxi\right)\}\!\!\}\|^2_{\Gamma_K}\leq Ch^2\sum_{i=\pm}\|\bu\|_{3,\Omega_i},\\
	&\sum_{K\in\mathcal T_h^{\Gamma}}\frac{h}{\lambda}\|\{\!\!\{\bn\times \left(\alpha\bm{\nabla}\times\nabla\times \bxi\right)\}\!\!\}\|^2_{\Gamma_K}\leq Ch^2\sum_{i=\pm}\|\bu\|_{3,\Omega_i}.
	\end{aligned}
	\end{equation*}
	For the term $h^{-2}\|\nabla_h\cdot \bxi\|^2_{\Omega_-\cup\Omega_+}$, we have
	\begin{equation*}
	\begin{aligned}
	h^{-2}\|\nabla_h\cdot \bxi\|^2_{\Omega_-\cup\Omega_+}&\leq Ch^{-2}\sum_{i=\pm}\|\bxi^i\|^2_{1,\Omega_i}\leq Ch^2\sum_{i=\pm}\|\bu\|^2_{3,\Omega_i}.	\end{aligned}
	\end{equation*} 
	The proof is completed by substituting these estimates into the definition of the discrete norm $\3bar\cdot\3bar_h$.
\end{proof}

We now establish the weak Galerkin orthogonality of the scheme \eqref{DisWeak}.
\begin{lemma}\label{Weakorthogonality}
Let $\bu\in \bm H^3(\Omega_-\cup\Omega_+)\cap \bm W$ represent the solution of the quad-curl interface problem (\ref{OriginProblem}), and let $\bu_h$ denote the solution of the numerical scheme (\ref{DisWeak}). Then, the following orthogonality relation holds
\begin{equation}
a_h(\bu-\bu_h,\bv_h)-J_1(\bm u_{h},\bm v_{h})-J_2(\bu_h,\bv_h)=0,\quad \forall \bv_h\in \boldsymbol{V}_h.
\end{equation}
\end{lemma}

The ghost penalty terms are demonstrated to satisfy weak consistency as follows
\begin{lemma}\label{WeakConsistency}
Let $\bu\in \bm H^3(\Omega_-\cup\Omega_+)$. For $\bm v_h\in \bm V_h$, the ghost penalty terms satisfy weak consistency, given by
\begin{equation}\label{ConsistencyResult}
J_1(\Pi_h^{\ast}\bm u,\bm v_{h})+J_2(\Pi_h^{\ast}\bu,\bv_h)\leq Ch\sum_{i=\pm}\|\bu\|_{3,\Omega_i}\3bar\bv_h\3bar_{\ast}.	
\end{equation}
\end{lemma}

\begin{proof}
Considering $\bm u\in \bm H^3(\Omega_-\cup\Omega_+)$, we have
\begin{align*}
&J_1(\Pi_h^{\ast}\bu,\bv_h)=\sum_{E\in\mathcal{E}_h^{\Gamma}}\sum_{l=0}^{4}\sum_{i=\pm}h^{2l-1}\int_{E}[\![\partial_n^l\Pi_h\bm u_{\mathcal E}^i]\!]\cdot[\![\partial_n^l\bv_h]\!]\d s\\
=&\sum_{E\in\mathcal{E}_h^{\Gamma}}\sum_{i=\pm}\left(\sum_{l=0}^2h^{2l-1}\int_{E}[\![\partial_n^l(\Pi_h\bm u_{\mathcal E}^i-\bm u_{\mathcal E}^i)]\!]\cdot[\![\partial_n^l\bv_h]\!]\d s+\sum_{l=3}^4h^{2l-1}\int_{E}[\![\partial_n^l\Pi_h\bm u_{\mathcal E}^i]\!]\cdot[\![\partial_n^l\bv_h]\!]\d s\right).
\end{align*}
By employing trace inequality \eqref{traceinq1}, inverse inequality \eqref{inverseinq}, and Lemma \ref{EstimateElem}, we obtain the following estimate
\begin{align*}
J_1(\Pi_h^{\ast}\bu,\bv_h)
\leq &\ C \sum_{l=0}^{2}\sum_{i=\pm}h^{l-2}\big(\|\Pi_h\bu_{\mathcal E}^i-\bu_{\mathcal E}^i\|_{l,\Omega_{h,i}}+h \|\Pi_h\bu_{\mathcal E}^i-\bu_{\mathcal E}^i\|_{l+1,\Omega_{h,i}}\big)\3bar\bv_h\3bar_{\ast}\\
& +C \sum_{l=3}^4\sum_{i=\pm}h\big(\|\Pi_h\bu_{\mathcal E}^i-\bu_{\mathcal E}^i\|_{3,\Omega_{h,i}}+ \|\bu_{\mathcal E}^i\|_{3,\Omega_{h,i}}\big)\3bar\bv_h\3bar_{\ast}\\
\leq &\ Ch\sum_{i=\pm}\|\bu\|_{3,\Omega_i}\3bar\bv_h\3bar_{\ast}.
\end{align*}
Similarly, utilizing the fact $[\![\nabla\times\bv_h]\!]=0$ on $E\in \mathcal{E}^\Gamma_h$, we can derive the inequality
\[J_2(\Pi_h^{\ast}\bu,\bv_h)\leq Ch\sum_{i=\pm}\|\bu\|_{3,\Omega_i}\3bar\bv_h\3bar_{\ast}.\]
Combining the above two estimates completes the proof.
\end{proof}

Now, we are prepared to present our main theoretical result
\begin{theorem}\label{EnergyEstimate}
Let $\bu$ and $\bu_h$ be the solutions of (\ref{OriginProblem}) and (\ref{DisWeak}), respectively. Suppose $\bu \in \bH^3(\Omega_-\cup\Omega_+)\cap \bm W$. Then, there exists a constant $C$ independent of $h$ such that
\begin{equation*}
\3bar \bu-\bu_h\3bar_h \leq C h\|\bu\|_{3,\Omega_-\cup \Omega_+}.
\end{equation*}
\end{theorem}
\begin{proof}
By the triangle inequality and \eqref{normrelation}, we have
\begin{equation}\label{Errortriangle}
\3bar\bu-\bu_h\3bar_h\leq\3bar\bu-\Pi_h^{\ast}\bu\3bar_h+C\3bar\Pi_h^{\ast}\bu-\bu_h\3bar_{\ast}.
\end{equation}
Lemma \ref{InterpolationEstimate} provides an estimate for the first term on the right-hand side above. It suffices to bound $\3bar\Pi_h^{\ast}\bu-\bu_h\3bar_{\ast}$
\begin{align*}
&c_1\3bar\Pi_h^{\ast}\bu-\bu_h\3bar_{\ast}^2
\leq \mathcal{A}_h(\bu_h-\Pi_h^{\ast}\bu,\bu_h-\Pi_h^{\ast}\bu)\\
=&\ a_h(\Pi_h^{\ast}\bu-\bu_h,\Pi_h^{\ast}\bu-\bu_h)+J_1(\Pi_h^{\ast}\bu-\bu_h,\Pi_h^{\ast}\bu-\bu_h)\\
&+J_2\big(\Pi_h^{\ast}\bu-\bu_h,\Pi_h^{\ast}\bu-\bu_h\big)\\
=&\ a_h(\Pi_h^{\ast}\bu-\bu,\Pi_h^{\ast}\bu-\bu_h)+J_1(\Pi_h^{\ast} \bu,\Pi_h^{\ast}\bu-\bu_h)\\
&+J_2\big(\Pi_h^{\ast} \bu,\Pi_h^{\ast}\bu-\bu_h\big)\\
\leq &\ C\3bar\bu-\Pi_h^{\ast}\bu\3bar_h\3bar\Pi_h^{\ast}\bu-\bu_h\3bar_{\ast}+J_1(\Pi_h^{\ast} \bu,\Pi_h^{\ast}\bu-\bu_h)\\
&\ +J_2\big(\Pi_h^{\ast} \bu,\Pi_h^{\ast}\bu-\bu_h\big)\\
\leq &\ Ch\sum_{i=\pm}\|\bu\|_{3,\Omega_i}\3bar\Pi_h^{\ast}\bu-\bu_h\3bar_{\ast},
\end{align*}
which completes the proof.
\end{proof}

\begin{remark}
Given the definition of $\3bar\cdot\3bar_h$, the convergence rate of $\|\nabla_h\cdot \bu\|_{\Omega_{-}\cup\Omega_+}$ is 2, thereby ensuring that the div-free condition is satisfied \end{remark}

\section{Condition number estimate}\label{Sec:conditionnumber}
In this section, following the approach in \cite{Ern2004,Ern2005}, we will provide an estimate for the condition number of the stiffness matrix associated with the finite element formulation \eqref{DisWeak}. Let $\{\bm{\phi}_i\}_{i=1}^N$ be a basis of the finite element space $\bm V_h$. Since the partition $\mathcal{T}_h$ is quasi-uniform, we can obtain the following estimate for $\bm w_h=\sum_{i=1}^NW_i\bm{\phi}_i$
\begin{lemma}\label{Vectortofunc}
There exist constants $C_1$ and $C_2$ such that
\begin{align}
\|\boldsymbol{w}_h\|_{\Omega_{h,-}\cup\Omega_{h,+}}&\leq C_2|W|_N,\label{vector2}\\
\|\boldsymbol{w}_h\|_{\Omega_{h,-}\cup\Omega_{h,+}}+\|\nabla\times\boldsymbol{w}_h\|_{\Omega_{h,-}\cup\Omega_{h,+}}&\geq h C_1|W|_N
\quad \forall \boldsymbol{w}_h\in \boldsymbol{V}_h,\label{vector1}
\end{align}
where $W=(W_1,...,W_N)^{\mathrm T}$, and $|W|_N$ denotes the Euclidean norm induced by the inner product $(W,V)_N=\sum_{i=1}^{N}W_iV_i$.
\end{lemma}
\begin{proof}
	We usually adopt the following transformation, $\bu\circ F_K = B_K^{-T}\hat{\bu}$, where the affine mapping $F_K(x) = B_K\hat{\boldsymbol{x}}+\boldsymbol{b}_K$. For a fixed element $K$, according to the Cauchy-Schwarz inequality, we have
\begin{align*}
	\|\bw\|^2_{K} &= \int_{K}(\sum_{i}^{13}W_i\phi_i)^2 \d \bm x=  \int_{K}(\sum_{i}^{13}W_i\phi_i)^2 \d \bm x  \leq\sum_{i=1}^{13}W_i^2 \sum_{i=1}^{13}\int_{K}\phi_i^2\d \bm x \nonumber\\
	&\leq \sum_{i=1}^{13}W_i^2 \sum_{i=1}^{13}\int_{\hat{K}}(B_K^{-T})^2\hat{\phi}_i^2|B_K| \d \hat{\boldsymbol{x}} \leq C\sum_{i=1}^{13}W_i^2.
	\end{align*}
Summing
 up the result for all elements $K\in\mathcal{T}_h$ completes the proof of the first inequality.
Noting that the relationship between $W$ and freedom of degree, i.e., $\left(W_1,W_2,...,W_{13}\right) = \left(\int_{e_1}\bw_h\cdot\bm \tau_1\d s, \int_{e_2}\bw_h\cdot\bm \tau_2\d s,..., \int_{K}\bw_h\cdot\bm x \d x \right)$, we have
\begin{align*}
	(\int_e \bw_h\cdot\bm \tau_i q\d s)^2 &\leq \int_e q^2\d s\int_e (\bw_h\cdot\bm \tau_i)^2\d s\leq C\|\bw_h\|^2_K\leq Ch^{-2}\|\bw_h\|^2_K,\quad \forall q\in P_2(K),\\
	(\nabla\times \bw_h(v_i))^2 &\leq \|\nabla\times\bw_h\|^2_{L^{\infty}(K)}\leq Ch^{-2}\|\nabla\times\bw_h\|^2_{K},\\
	(\int_{K}\bw_h\cdot\bm x \d \bm x)^2&\leq \int_{K}\bw_h^2\d \bm x\int_{K} \bm x^2 \d \bm x\leq C\|\bw_h\|_K^2\leq Ch^{-2}\|\bw_h\|^2_K.
\end{align*}
Summing up the results for all elements $K\in\mathcal{T}_h$ concludes the
proof of the second inequality.
\end{proof}
For a matrix $A$, we define the operator norm as
\begin{equation*}
\|A\|=\sup_{W\in\mathbb{R}^N}\frac{|AW|_N}{|W|_N}.
\end{equation*}
Furthermore, the condition number of matrix $A$ is defined by
\begin{equation}
\mathcal{K}(A):=\|A\|\|A^{-1}\|.
\end{equation}

The following inverse estimate and a Poincar\'e inequality for the appropriate norms are the key points for estimating the condition number.
\begin{lemma}\label{Disinverse}
For all $\bv_h\in\boldsymbol{V}_h$, there exists a constant $C_I>0$ such that
\begin{equation}
\3bar \bv_h\3bar_{\ast} \leq C_Ih^{-2}\|\bv_h\|_{\Omega_{h,-}\cup\Omega_{h,+}}.
\end{equation}
\end{lemma}

\begin{proof}
The above inequality is derived from trace inequalities (\ref{TraceResult}) and \eqref{traceinq1}, and the inverse inequality \eqref{inverseinq}.
\end{proof}

\begin{lemma}\label{Dispoin}
For all $\bv_h\in\boldsymbol{V}_h$, it holds
\begin{equation}
\3bar \bv_h\3bar_{\ast}\geq C_{P}(\|\bu_h\|_{\Omega_{h,-}\cup\Omega_{h,+}}+\|\nabla\times\bv_h\|_{\Omega_{h,-}\cup\Omega_{h,+}}).
\end{equation}
\end{lemma}

\begin{proof}
It follows from Lemma \ref{ghostestimate} and Lemma \ref{poincare} that
\begin{align*}
\|\nabla\times\bv_h\|_{\Omega_{h,-}\cup\Omega_{h,+}}^2
&\leq C\Big(\|\nabla\times\bv_h\|^2_{\Omega_-\cup\Omega_+}+h^2J_2(\bm v_h,\bm v_h)\Big)\nonumber\\
&\leq C\Big(\|\bm{\nabla}\times\nabla\times\bv_h\|^2_{\Omega_-\cup\Omega_+}+\mathcal{G}_2(\bv_h,\bv_h)+h^2J_2(\bm v_h,\bm v_h)\Big),
\end{align*}
and 
\begin{align*}
	C\|\bv_h\|_{\Omega_{h,-}\cup\Omega_{h,+}}\leq \|\bm{\nabla}\times\nabla\times\bv_h\|^2_{\Omega_-\cup\Omega_+}+\|\nabla_h\cdot\bv_h\|^2_{\Omega_-\cup\Omega_+}+\sum_{i=0}^3\mathcal G_i(\bm v_h,\bm v_h)+h^2J_1(\bm v_h,\bm v_h).
\end{align*}
Recalling that the fact $\|\bm\nabla\times\nabla\times\bm v_h\|_{\Omega_{-}\cup\Omega_{+}}\leq \|\bm\nabla\times\nabla\times\bm v_h\|_{\Omega_{h,-}\cup\Omega_{h,+}}$ and the definition of $\3bar\cdot\3bar_{\ast}$, we have
\begin{equation*}
\3bar\bu_h\3bar_{\ast}\geq C_p(\|\bv_h\|_{\Omega_{h,-}\cup\Omega_{h,+}}+\|\nabla\times\bv_h\|_{\Omega_{h,-}\cup\Omega_{h,+}}).
\end{equation*}
which completes the proof.
\end{proof}

Now, we are in a position to present the main result of this section.
\begin{theorem}\label{conditionnumber}
The condition number of the stiffness matrix $A$ satisfies the following estimate
\begin{equation}
\mathcal{K}(A)\leq \frac{c_2C_2^2C_I^2}{c_1C_1^2C_P^2}h^{-6}.
\end{equation}
\end{theorem}

\begin{proof}
By the definition, we have
\begin{equation}\label{defAW}
|AW|_N=\sup_{V\in\mathbb{R}^N}\frac{(AW,V)_N}{|V|_N}=\sup_{V\in\mathbb{R}^N}\frac{\mathcal{A}_h(\boldsymbol{w}_h,\bv_h)}{|V|_N}.
\end{equation}
Using \eqref{continuity} and Lemma \ref{Disinverse}, we obtain
\begin{equation}\label{normA}
\begin{aligned}
\mathcal{A}_h(\boldsymbol{w}_h,\bv_h)\leq c_2\3bar \boldsymbol{w}_h\3bar_{\ast}\3bar \bv_h\3bar_{\ast}&\leq c_2C_I^2h^{-4}\|\boldsymbol{w}_h\|_{\Omega_{h,-}\cup\Omega_{h,+}}\|\bv_h\|_{\Omega_{h,-}\cup\Omega_{h,+}}\\
&\leq c_2 C^2_2C_I^2h^{-4}|W|_N|V|_N,
\end{aligned}
\end{equation}
which implies $\|A\|=\sup_{W\in\mathbb{R}^N}\frac{|AW|_N}{|W|_N}\leq c_2 C^2_2C_I^2h^{-4}$.

Applying \eqref{coercivity} and Lemma \ref{Dispoin}, we get
\begin{equation}\label{normA1}
\begin{aligned}
\mathcal{A}_h(\boldsymbol{w}_h,\boldsymbol{w}_h)\geq c_1\3bar \boldsymbol{w}_h\3bar_{\ast}\3bar \bw_h\3bar_{\ast}&\geq{c_1C_P^2}\left(\|\boldsymbol{w}_h\|_{\Omega_{h,-}\cup\Omega_{h,+}}+\|\nabla\times\bw_h\|_{\Omega_{h,-}\cup\Omega_{h,+}} \right)^2\\
&\geq {c_1C_P^2C_1^2}h^2|W|_N|W|_N,
\end{aligned}
\end{equation}
which, together with \eqref{defAW}, leads to
\begin{equation}
\begin{aligned}
|AW|_N\geq {c_1C_P^2C_1^2}h^2|W|_N.
\end{aligned}
\end{equation}
Taking $W=A^{-1}Y$ yields
\begin{equation}
\|A^{-1}\|\leq \frac{1}{c_1C_P^2C_1^2}h^{-2}.
\end{equation}
From the definition of $\mathcal K(A)$, we can obtain the desired result
\[\mathcal K(A)\leq \frac{c_2C_2^2C_I^2}{c_1C_1^2C_P^2}h^{-6}.\]
\end{proof}

\begin{remark}
From the proof of Lemma \ref{Dispoin}, the terms $h^{-2}\int_{\Omega_-\cup\Omega_+}(\nabla_h\cdot \bu_h) (\nabla_h\cdot \bv_h) \d A$ and $\mathcal G_0$ can improve the constant $C_P$, and hence, reduce $\mathcal K(A)$ from Theorem \ref{conditionnumber}.
\end{remark}

\begin{remark}
	It is worth mentioning that the condition number of the stiffness matrix is determined by the degrees of freedom of the curlcurl conforming elements and the quad-curl problem itself. When addressing the quad-curl problem rather than the quad-curl interface problem, the condition number of the matrix remains on the order of $O(h^{-6})$.
\end{remark}

\section{Numerical experiments}\label{Sec:Numerical}
In this section, we conduct two-dimensional numerical tests to validate the convergence rate of the proposed scheme. In the subsequent examples, we adopt the following values of $\kappa_1$ and $\kappa_2$ for an interface element $K$,
\begin{align*}
	\kappa_1 = \frac{\alpha_+|K_-|}{\alpha_+|K_-|+\alpha_-|K_+|},\quad \kappa_2 = \frac{\alpha_-|K_+|}{\alpha_+|K_-|+\alpha_-|K_+|},
\end{align*}
where $K_i = \Omega_i\cap K$ and $|K_i|$ is the area of $K_i$. For simplicity, we introduce the notation $\|\cdot\|_h = \sqrt{\sum_{i=\pm}\|\cdot\|_{\Omega_{i}}^2}$.

	\begin{figure}[h]
		\centering
		\subfigure{
			\begin{minipage}[t]{0.41\linewidth}
				\centering
				\includegraphics[width=0.87\textwidth]{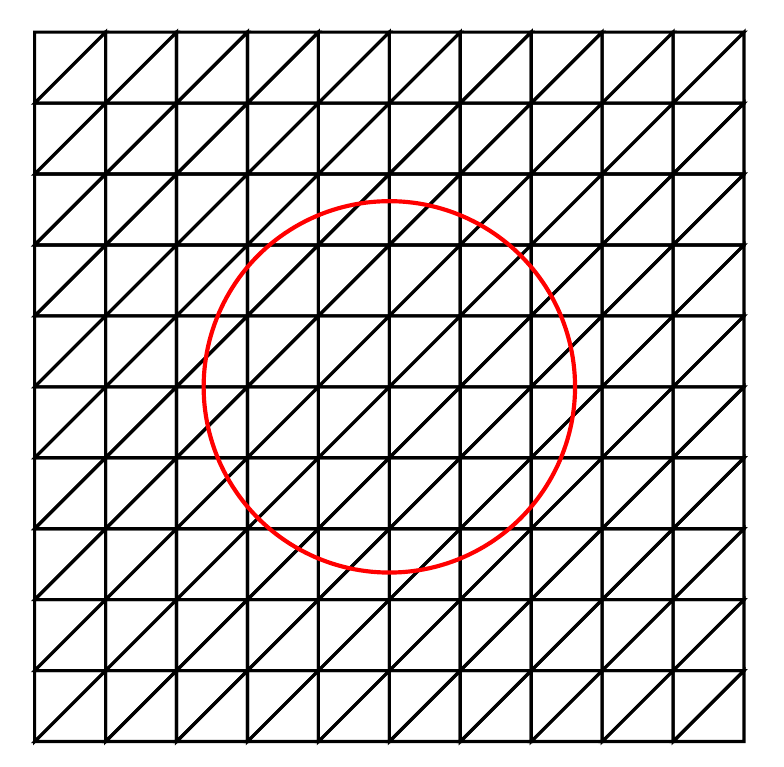}
			\end{minipage}
		}
		\subfigure{
			\begin{minipage}[t]{0.42\linewidth}
				\centering
				\includegraphics[width=0.87\textwidth]{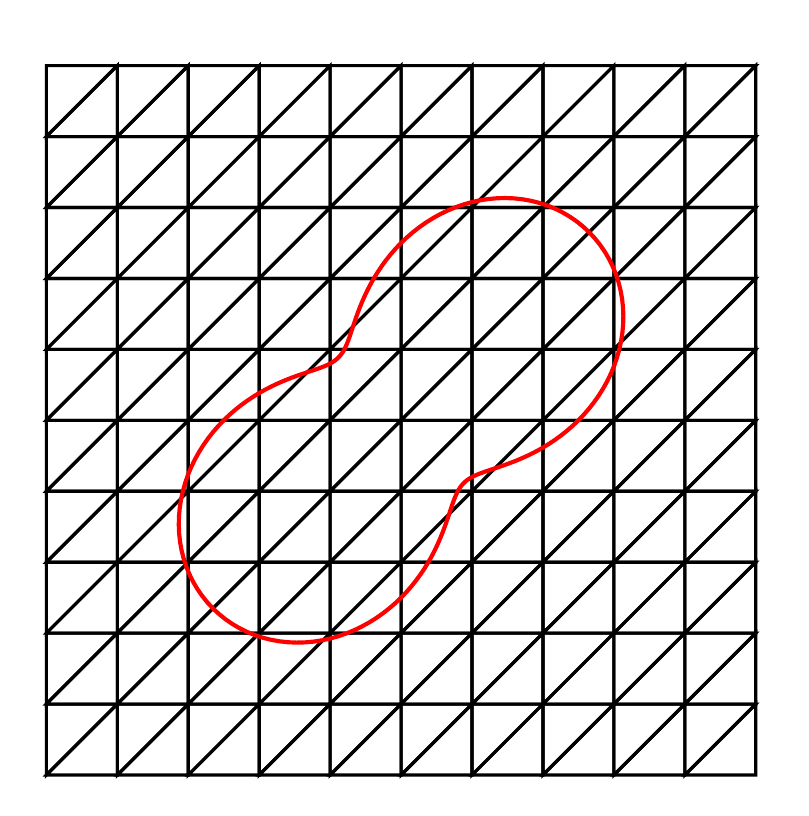}
			\end{minipage}
		}
		\caption{The circle-shape interface (left) and the peanut-shape interface (right).}
		\label{InterfaceShape}
	\end{figure}

	\begin{table}[htb]\centering\small
		\renewcommand\arraystretch{1.2}
		\caption{Numerical results for example 1 when  $\alpha_-=\alpha_+=1$, $\gamma=1$.}
		\setlength{\tabcolsep}{1.3mm}{
			\begin{tabular}{ccccccccc}
				\hline
				$h$&      $\|\boldsymbol{e}_h\|_{h}$&  Rates& $\|\nabla\times\boldsymbol{e}_h\|_{h}$&      Rates&   $\|(\nabla\times)^2\boldsymbol{e}_h\|_{h}$&Rates&$\|\nabla_h\cdot\bm e_h\|_{h}$ &Rates\\ 
				\hline
				20& 7.4869e-01&    --&  8.1887e+00&     --&   1.2171e+02&       --& 8.4885e-01&--\\
				40& 1.8048e-01& 2.05&  2.0853e+00&  1.97&  5.4720e+01&   1.15& 2.5086e-01&1.76 \\
				80& 3.6680e-02& 2.30&  4.2248e-01&  2.30&  2.5819e+01&    1.08& 6.6711e-02& 1.91\\
				160& 7.6040e-03& 2.27&  8.4446e-02&  2.32&  1.2710e+01&   1.02&1.6919e-03 &1.98\\
				\hline
			\end{tabular}
		}
		\label{Table_con1}
		\vspace{12pt}
	\end{table}
	\begin{table}[htb]\centering\small
		\renewcommand\arraystretch{1.2}
		\caption{Numerical results for example 1 when  $\alpha_-=1,\ \alpha_+=2$, $\gamma=1$.}
		\setlength{\tabcolsep}{1.3mm}{
			\begin{tabular}{ccccccccc}
				\hline
				$h$&      $\|\boldsymbol{e}_h\|_{h}$&  Rates& $\|\nabla\times\boldsymbol{e}_h\|_{h}$&      Rates&   $\|(\nabla\times)^2\boldsymbol{e}_h\|_{h}$&Rates&$\|\nabla_h\cdot\bm e_h\|_{h}$ &Rates\\ 
				\hline
				20& 6.8571e-01&    --&  7.0294e+00&     --&   1.1261e+02&       --& 8.8463e+00&--\\
				40& 1.6138e-01& 2.04&  1.7147e+00&  1.99&  5.2760e+01&   0.99&  2.5118e-01&2.01 \\
				80& 3.3756e-02& 2.13&  3.5794e-01&  2.12&  2.5593e+01&    1.01&  6.6357e-02& 1.96\\
				160& 7.3717e-03& 2.08&  7.6084e-02&  2.10&  1.2683e+01&   1.01&1.6913e-02 &1.97\\
				\hline
			\end{tabular}
		}
		\label{Table_con2}
		\vspace{12pt}
	\end{table}
	\begin{table}[htb]\centering\small
		\renewcommand\arraystretch{1.2}
		\caption{Numerical results for example 1  when  $\alpha_-=1,\ \alpha_+=10$, $\gamma=1$.}
		\setlength{\tabcolsep}{1.3mm}{
			\begin{tabular}{ccccccccc}
				\hline
			$h$&      $\|\boldsymbol{e}_h\|_{h}$&  Rates& $\|\nabla\times\boldsymbol{e}_h\|_{h}$&      Rates&   $\|(\nabla\times)^2\boldsymbol{e}_h\|_{h}$&Rates&$\|\nabla_h\cdot\bm e_h\|_{h}$ &Rates\\ 
				\hline
				20& 5.5723e-01&    --&  5.0779e+00&     --&   1.0215e+02&       --& 1.0397e+00&--\\
				40& 1.3596e-01& 2.04&  1.2783e+00&  1.99&  5.1364e+01&   0.99& 2.5771e-01&2.01 \\
				80& 3.0980e-02& 2.13&  2.9407e-01&  2.12&  2.5433e+01&    1.01&  6.6357e-02& 1.96\\
				160& 7.3265e-03& 2.08&  6.8425e-02&  2.10&  1.2655e+01&   1.01&1.6913e-02 &1.97\\
				\hline
			\end{tabular}
		}
		\label{Table_con3}
		\vspace{12pt}
	\end{table}
	\begin{table}[htb]\centering\small
		\renewcommand\arraystretch{1.2}
		\caption{Numerical results for example 1  when  $\alpha_-=1,\ \alpha_+=100$, $\gamma=1$.}
		\setlength{\tabcolsep}{1.3mm}{
			\begin{tabular}{ccccccccc}
				\hline
				$h$&      $\|\boldsymbol{e}_h\|_{h}$&  Rates& $\|\nabla\times\boldsymbol{e}_h\|_{h}$&      Rates&   $\|(\nabla\times)^2\boldsymbol{e}_h\|_{h}$&Rates&$\|\nabla_h\cdot\bm e_h\|_{h}$ &Rates\\ 
				\hline
				20& 5.0266e-01&    --&  4.3608e+00&     --&   1.0063e+02&       --& 1.1272e+00&--\\
				40& 1.2842e-01& 1.97&  1.1656e+00&  1.90&  5.1248e+01&   0.97& 2.6023e-01&2.11 \\
				80& 3.0461e-02& 2.08&  2.7949e-01&  2.06&  2.5403e+01&    1.01&  6.6437e-02& 1.97\\
				160& 7.4058e-03& 2.04&  6.6803e-02&  2.06&  1.2647e+01&   1.01&1.6920e-02&1.97\\
				\hline
			\end{tabular}
		}
		\label{Table_con4}
		\vspace{12pt}
	\end{table}

\subsection{Example 1}\label{Examp1}
We consider the quad-curl interface equation \eqref{OriginProblem} with homogeneous boundary conditions.
Let the domain $\Omega=[-1,1]\times[-1,1]$, and the interface be the circle with center origin and radius $\frac{\pi}{6}$ (see Fig. \ref{InterfaceShape}). We determine the source term and the jump conditions $\varphi_3$, $\varphi_4$ according to the exact solution given below
\begin{equation}       
\boldsymbol{u}^+=\boldsymbol{u}^-=\left(                
\begin{array}{c}   %
3\pi\sin^3(\pi x)\sin^2(\pi y)\cos(\pi y)\\  
-3\pi\sin^3(\pi y)\sin^2(\pi x)\cos(\pi x)\\  
\end{array}
\right),  
\end{equation}
We present the numerical results for four groups of parameters: $\alpha_-=1$ and $\alpha_+=1, 2, 10$, or $100$, and $\gamma=1$. Additionally, the penalty parameter is set as $\lambda=100$.	
The finite element solution is denoted by $\bu_h$. To measure the error between the exact solution and the finite element solution, we denote
\begin{equation*}
\boldsymbol{e}_h=\bu-\bu_h.
\end{equation*}
Tables \ref{Table_con1}-\ref{Table_con4} illustrate various errors and convergence rates in different norms. From the tables, we can observe second-order convergence in $\|\bu-\bu_h\|$, $\|\nabla\times(\bu-\bu_h)\|$, and $\|\text{div}_h(\bu-\bu_h)\|$ for different values of $\alpha_+$. Additionally, we also get first-order convergence in $\|\nabla\times\nabla\times(\bu-\bu_h)\|$. These results are consistent with the theoretical analysis (Theorem \ref{EnergyEstimate}).

	\begin{figure}[h]
		\centering
		\subfigure[$\sum_i\|\bu -\bu_h\|_{\Omega_{i}}$]{
			\begin{minipage}[t]{0.47\linewidth}
				\centering
				\includegraphics[width=\textwidth]{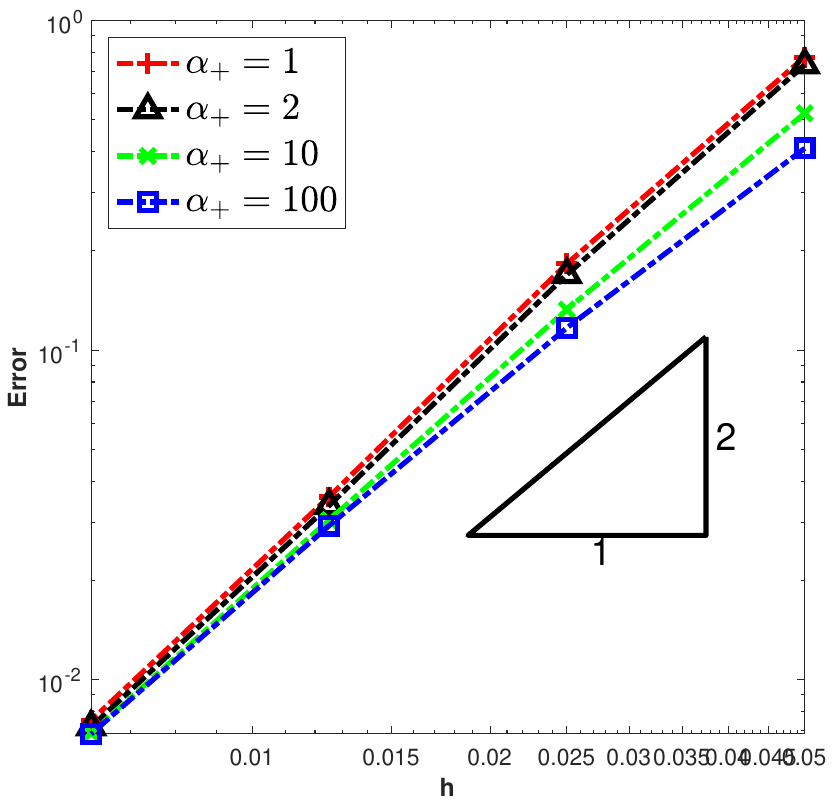}
			\end{minipage}
			\label{fig1b}
		}
		\subfigure[$\sum_i\|\text{curl}(\bu -\bu_h)\|_{\Omega_{i}}$]{
			\begin{minipage}[t]{0.47\linewidth}
				\centering
				\includegraphics[width=\textwidth]{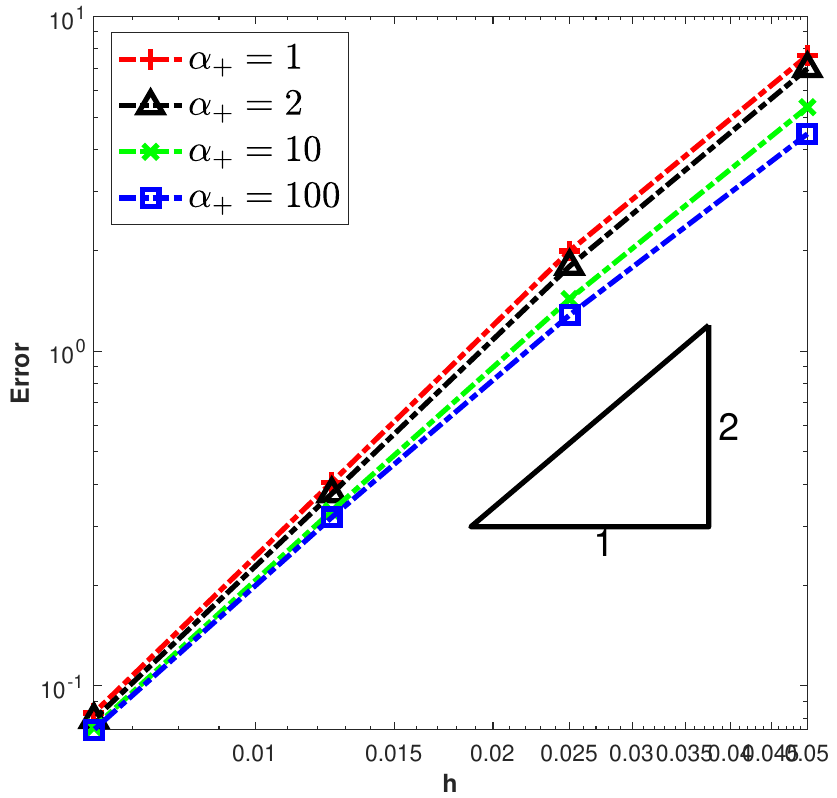}
			\end{minipage}
			\label{fig1a}
		}
		\subfigure[$\sum_i\|\text{curlcurl}(\bu -\bu_h)\|_{\Omega_{i}}$]{
			\begin{minipage}[t]{0.47\linewidth}
				\centering
				\includegraphics[width=\textwidth]{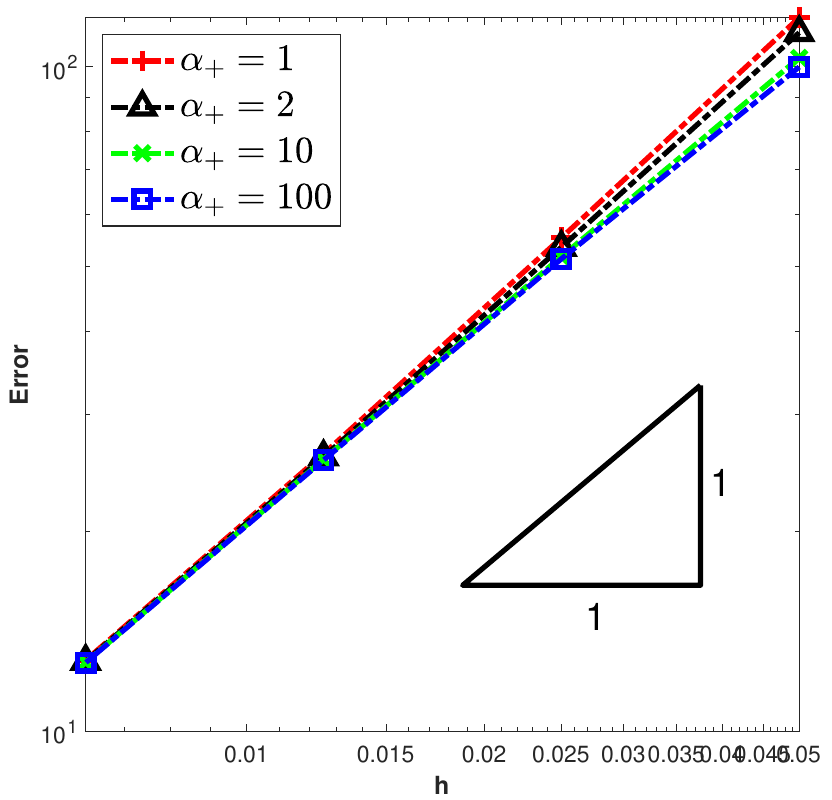}
			\end{minipage}
			\label{fig1c}
		}
		\subfigure[$\sum_i\|\text{div}_h(\bu -\bu_h)\|_{\Omega_{i}}$]{
			\begin{minipage}[t]{0.47\linewidth}
				\centering
				\includegraphics[width=\textwidth]{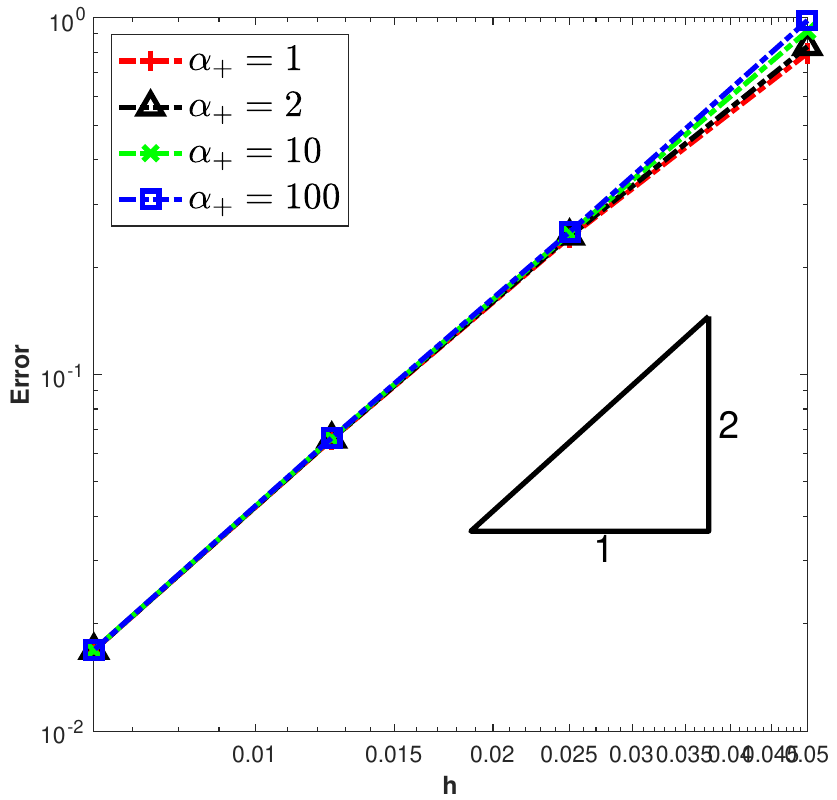}
			\end{minipage}
			\label{fig1d}
		}
		\caption{Numerical results for example 2 when $\alpha_+=1,\ 2,\ 10,\ 100$ and $\gamma=0.01$.}
		\label{Converfig1}
	\end{figure}

\subsection{Example 2}
In this example, we solve a problem on the domain $\Omega=[-1,1]\times[-1,1]$ with a peanut-shaped interface defined by the following parametric function
\begin{align*}
x=(\frac{1}{2}+\frac{\sin(2\theta)}{4})\cos(\theta),\\
y=(\frac{1}{2}+\frac{\sin(2\theta)}{4})\sin(\theta),
\end{align*}
where $0\leq\theta\leq 2\pi$ (see Fig. \ref{InterfaceShape}). The exact solution is the same as in Example \ref{Examp1}. We take the parameters $\alpha_-=1$ and $\alpha_+=1, 2, 10$ or $100$, and $\gamma=0.01$. Other parameters are the same as in Example \ref{Examp1}. Fig. \ref{Converfig1} demonstrates numerical errors in various norms, confirming our theoretical results.

\subsection{Example 3}
	In this particular instance, we investigate a problem within the same interface and domain as described in Example \ref{Examp1}. However, in this case, we only provide the right-hand side term $\bm f$ and the jump conditions, omitting the exact solution, which is closer to an actual problem scenario. We set the source term and the jump conditions as follows: $\bm f^+ = (10,0)^T,\ \bm f^-=(1,0)$, $\ \varphi_3 =2x+3y$, $\varphi_4 = -9y$ and use $\bm e_h= \bu_h-\bu_{\frac{h}{2}}$ to assess the convergence order in different norms. The numerical results are displayed in Table \ref{Table_examp3} with the parameters $\alpha_-=\alpha_+=1,\ \gamma=1$, and $\lambda=100$. The detected convergence rates are consistent with the theoretical analysis.
	\begin{table}[htb]\centering\small
		\renewcommand\arraystretch{1.2}
		\caption{Numerical results for example 3 when  $\alpha_-=1,\ \alpha_+=1$, $\gamma=1$.}
		\setlength{\tabcolsep}{1.3mm}{
			\begin{tabular}{ccccccccc}
				\hline
				$h$&      $\|\boldsymbol{e}_h\|_{h}$&  Rates& $\|\nabla\times\boldsymbol{e}_h\|_{h}$&      Rates&   $\|(\nabla\times)^2\boldsymbol{e}_h\|_{h}$&Rates&$\|\nabla_h\cdot\bm e_h\|_{h}$ &Rates\\ 
				\hline
				10& 1.8205e-02&    --&  5.2201e-02&     --&   6.5244e-01&       --& 9.6489e-03&--\\
				20& 4.3127e-03& 2.07&  1.4381e-02&  1.86&  3.2826e-01&   0.99& 2.9448e-03&1.71 \\
				40& 1.0120e-03& 2.09&  3.1824e-01&  2.18&  1.5992e-01&    1.04&  7.9537e-04& 1.89\\
				80& 2.5019e-04& 2.02&  6.7956e-02&  2.23&  7.8548e-02&   1.03&2.0710e-04&1.94\\
				\hline
			\end{tabular}
		}
		\label{Table_examp3}
		\vspace{12pt}
	\end{table}
\subsection{Example 4}
	In this subsection, our primary focus is to verify the robustness of the numerical scheme under the condition of $\gamma=0$. We take the following exact solution
	\begin{equation*}      
		\boldsymbol{u}^+=\boldsymbol{u}^-=\left(                
		\begin{array}{c}   %
		(x^2-1)^3(y^2-1)^2y\\  
		-(x^2-1)^2(y^2-1)^2x\\  
		\end{array}
		\right),  
	\end{equation*}
and the source term and the jump conditions $\varphi_3$, $\varphi_4$ can be calculated based on the exact solution. The domain, interface, and penalty parameter $\lambda$ remain consistent with the description in Example \ref{Examp1}. Table \ref{Table_con6}-\ref{Table_con7} present numerical results for the cases where $\alpha^+ = 10,100$, $\alpha_-=1$, along with $\gamma = 0$. The observed convergence order consistently corresponds to the theoretical analysis, verifying the robustness of the proposed numerical scheme.
\begin{table}[htb]\centering\small
	\renewcommand\arraystretch{1.2}
	\caption{Numerical results for example 4  when  $\alpha_-=1,\ \alpha_+=10$, $\gamma=0$.}
	\setlength{\tabcolsep}{1.3mm}{
		\begin{tabular}{ccccccccc}
			\hline
			$h$&      $\|\boldsymbol{e}_h\|_{h}$&  Rates& $\|\nabla\times\boldsymbol{e}_h\|_{h}$&      Rates&   $\|(\nabla\times)^2\boldsymbol{e}_h\|_{h}$&Rates&$\|\nabla_h\cdot\bm e_h\|_{h}$ &Rates\\ 
			\hline
			20& 1.4566e-02&    --&  5.9715e-02&     --&   1.1735e+00&       --& 1.0316e-02&--\\
			40& 3.2612e-03& 2.16&  1.3486e-02&  2.15&  5.8494e-01&   1.00& 2.5016e-03&2.04 \\
			80& 7.7849e-04& 2.07&  3.2357e-03&  2.06&  2.9219e-01&    1.00&  6.3291e-04& 1.98\\
			160& 1.9006e-04& 2.03&  7.8940e-04&  2.04&  1.4594e-01&   1.00&1.5986e-04&1.99\\
			\hline
		\end{tabular}
	}
	\label{Table_con6}
	\vspace{12pt}
\end{table}
\begin{table}[htb]\centering\small
	\renewcommand\arraystretch{1.2}
	\caption{Numerical results for example 4  when  $\alpha_-=1,\ \alpha_+=100$, $\gamma=0$.}
	\setlength{\tabcolsep}{1.3mm}{
		\begin{tabular}{ccccccccc}
			\hline
			$h$&      $\|\boldsymbol{e}_h\|_{h}$&  Rates& $\|\nabla\times\boldsymbol{e}_h\|_{h}$&      Rates&   $\|(\nabla\times)^2\boldsymbol{e}_h\|_{h}$&Rates&$\|\nabla_h\cdot\bm e_h\|_{h}$ &Rates\\ 
			\hline
			20& 1.3969e-02&    --&  5.6955e-02&     --&   1.1669e+00&       --& 1.0751e-02&--\\
			40& 3.4940e-03& 2.00&  1.4317e-02&  1.99&  5.8384e-01&   1.00& 2.5152e-03&2.10 \\
			80& 8.7304e-04& 2.00&  3.5905e-03&  2.00&  2.9197e-01&    1.00&  6.3324e-04& 1.99\\
			160& 2.2012e-04& 1.99&  9.0335e-04&  1.99&  1.4587e-01&   1.00&1.5993e-04&1.99\\
			\hline
		\end{tabular}
	}
	\label{Table_con7}
	\vspace{12pt}
\end{table}
\subsection{Condition Number}
In this investigation, we delve into the condition number analysis of the resulting stiffness matrix, an essential step in assessing its numerical stability. The calculated condition number, denoted as $\mathcal{K}(A)$, along with the magnitudes of $|A|$ and $|A^{-1}|$, pertaining to the sparse matrices derived from Example 3, are presented in Table \ref{Table_cond}. Notably, it becomes evident that as the mesh size diminishes toward zero, $\mathcal{K}(A)$, $|A|$, and $|A^{-1}|$ exhibit growth order of $O(h^{-6})$, $O(h^{-4})$, and $O(h^{-2})$, respectively. This observation aligns seamlessly with the theoretical analysis in Theorem \ref{conditionnumber}. 

	\begin{table}[htb]\centering\small
	\renewcommand\arraystretch{1.2}
	\caption{Condition number and matrix norm.}
	\setlength{\tabcolsep}{2.0mm}{
		\begin{tabular}{ccccccc}
			\hline
			$h$&      $\mathcal{K}(A)$&  Rates& $\|A\|$&      Rates&   $\|A^{-1}\|$&Rates\\ 
			\hline
			10& 2.9831e+10&    --&  8.2290e+09&     --&   3.6251e+00&       --\\
			20& 1.6108e+12& -5.75&  1.2802e+11&  -3.96&  1.2582e+01&   -1.80 \\
			40& 8.1416e+13& -5.66&  2.0382e+12&  -3.99&  3.9946e+01&    -1.67\\
			80& 4.6259e+15& -5.83&  3.2572e+13&  -4.00&  1.4202e+02&   -1.83\\
			\hline
		\end{tabular}
	}
	\label{Table_cond}
	\vspace{12pt}
\end{table}

\section{Conclusion}\label{Sec:conclusion} In this paper, we consider the quad-curl interface problem arising in Magnetohydrodynamics and Maxwell transmission problems with heterogeneous materials. We introduce the weak formulation for the quad-curl interface problem and demonstrate its well-posedness. To address the challenges posed by the interface, we propose an unfitted finite element method using the $H(\curl^2)$ conforming element and the idea of unfitted Nitsche's method without explicitly fitting the interface. We prove the optimal convergence rate for the numerical methods and validate the theoretical results with numerical examples. 

\section*{Acknowledgments}
This work is supported in part by the Andrew Sisson Fund, Dyason Fellowship, the Faculty Science Researcher Development Grant of the University of Melbourne,  the National Natural Science Foundation of China grants NSFC 12131005.

\section*{Declaration of competing interest}
The authors declare that they have no known competing financial interests or personal relationships that could have appeared to influence the work reported in this paper.
\bibliographystyle{siamplain}
\bibliography{references}

\end{document}